\documentclass[reqno]{amsart}
\usepackage{amsmath, amssymb, amsthm, epsfig}
\usepackage{hyperref, latexsym}
\usepackage{url}
\usepackage[mathscr]{euscript}

\usepackage{color}
\usepackage{fullpage} 
\usepackage{setspace}

\def\today{\ifcase\month\or
  January\or February\or March\or April\or May\or June\or=
  July\or August\or September\or October\or November\or December\fi
  \space\number\day, \number\year}

 \newtheorem{theorem}{Theorem}
  
 \newtheorem{lemma}[theorem]{Lemma}
 
 \newtheorem{corollary}[theorem]{Corollary}
 \theoremstyle{definition}

 \theoremstyle{remark}
 \newtheorem{remark}[theorem]{Remark}

 \newcommand{\R}{\mathbb{R}}
  
 \newcommand{\N}{\mathbb{N}}
 
 \newcommand{\Z}{\mathbb{Z}}

\newcommand{\var}{{\rm Var\,}}

\begin{document}

\title[Sharp inequalities for maximal operators on finite graphs]{Sharp inequalities for maximal operators on finite graphs}
\author[Gonz\'{a}lez-Riquelme and Madrid]{Cristian Gonz\'{a}lez-Riquelme and Jos\'e Madrid}
\date{\today}

\address{IMPA - Instituto de Matem\'{a}tica Pura e Aplicada\\
Rio de Janeiro - RJ, Brazil, 22460-320.}
\email{cristian@impa.br}

\address{Department of  Mathematics,  University  of  California,  Los  Angeles (UCLA),  Portola Plaza 520, Los  Angeles,
California, 90095, USA}
\email{jmadrid@math.ucla.edu}

\allowdisplaybreaks
\numberwithin{equation}{section}

\maketitle

\begin{abstract}Let $G=(V,E)$ be a finite graph (here $V$ and $E$ denote the set of vertices and edges of $G$ respectively)  and $M_G$ be the centered Hardy-Littlewood maximal operator defined there. We find the optimal value ${\bf{C}}_{G,p}$ such that the inequality 
$$\var_{p}M_{G}f\le {\bf C}_{G,p}\var_{p}f$$
holds for every $f:V\to \mathbb{R},$ where $\var_p$ stands for the $p$-variation, when: 
(i) $G=K_n$ (complete graph) and $p\in [\frac{\log(4)}{\log(6)},\infty)$ or $G=K_4$ and $p\in (0,\infty)$; (ii) $G=S_n$(star graph) and $1\ge p\ge \frac{1}{2}$; $p\in (0,\frac{1}{2})$ and $n\ge C(p)$ or $G=S_3$ and $p\in (1,\infty).$ 
We also find value of the norm $\|M_{G}\|_{2}$
when: (i) $G=K_n$ and $n\ge 3$; (ii) $G=S_n$ and $n\ge 3.$
\end{abstract}

\section{Introduction}

\subsection{A brief historical overview and background}
The study of maximal operators is a central theme in analysis. Since the beginning of the past century many properties of these operators have been useful in several areas of mathematics. In general, properties related with the behavior of the norm of these operators have been the main interest of study, until the work of Kinnunen \cite{Ki} where he observed that it was possible to prove the boundedness of the map $$f\to Mf$$ from $W^{1,p}(\mathbb{R}^d)\to W^{1,p}(\mathbb{R}^d),$ when $p>1$, (where $M$ stands for the centered Hardy-Littlewood maximal function). He also showed meaningful applications of this property. This work was the first to study maximal operators at a derivative level. Since then many authors followed this path and proved several results concerning these {\it{derivative level questions}} in a broad class of contexts and for several kinds of maximal operators. This topic of harmonic analysis has been named {\it{regularity theory of maximal operators}}. For general information on this area there is an interesting survey \cite{CaSurvey} by Carneiro.
    

An interesting framework of study is the following. Let $G=(V,E)$ be a graph and $f:V\to\R$ a real valued function. We define the Hardy-Littlewood maximal function of $f$ along $G$ at the point $e\in V$ by 
\begin{equation}\label{classical}
    M_{G}f(e):=\max_{r\geq0}\frac{1}{|B(e,r)|}\sum_{m\in B(e,r)}|f(m)|,
\end{equation} 
where $B(e,r)=\{m\in V;d_{G}(e,m)\leq r\}$, where $d_{G}$ is the metric induced by the edges of $G$ (that is, the distance between two vertices is the number of edges in a shortest path connecting them). A more general version of this, is the so called fractional maximal function defined by
\begin{equation*}
    M_{\alpha,G}f(e):=\max_{r\geq0}\frac{1}{|B(e,r)|^{1-\alpha}}\sum_{m\in B(e,r)}|f(m)|
\end{equation*}
for all $\alpha\in(0,1]$. Both operators have uncentered versions defined by 
\begin{equation*}
    \widetilde{M}_{\alpha,G}f(e)=\max_{B(v,r)\ni e}\frac{1}{|B(v,r)|^{1-\alpha}}\sum_{m\in B(v,r)}|f(m)|
\end{equation*}
 for the fractional one, and $\widetilde{M}_{G}=\widetilde{M}_{0,G}$ for the classical one.  In this paper we study the regularity properties of these objects acting on $l^p-$spaces and bounded $p-$variation spaces. We focus on the classical maximal function defined in \eqref{classical}. 

{In recent years there has been a lot of interest in studying discrete analogous to classical results in harmonic analysis. In particular, in problems involving maximal operators acting in different settings, whether in a discrete context (see, for instance \cite{CaHu,Ma,Ma2,Te}) or as an intermediary step towards the solution of a continuous problem (see, for instance \cite{Me}). Often, these results depends strongly on the structure of the domain considered, as we will observe, in our main theorems the geometry of the graphs analyzed plays a fundamental role.}

The most natural context for the discrete version of the {\it derivative level} questions mentioned above is the following, given $p\in (0,\infty)$ we define the $p$-variation of a function $f:V\to \mathbb{R}$ as follows
$$
\var_{p}f:=\left(\frac{1}{2}\sum_{n}\sum_{\substack{m \\ d_{G}(n,m)=1}}|f(n)-f(m)|^p\right)^{1/p}.
$$
The first work to address a result concerning the derivative level (in this case, the variation) of a maximal operator in a discrete setting was \cite{BCHP}; where they, among other things, found sharp constant for the $1$-variation of the uncentered Hardy-Littlewood maximal operator $\widetilde{M}_{\mathbb{Z}}$, where in $\mathbb{Z}$ we take the usual distance (observe that naturally we can see $\Z$ as an infinite linear tree). That is, they proved that for every $f:\mathbb{Z}\to \mathbb{R},$ we have 
$$\var_{1}\widetilde{M}_{\mathbb{Z}}f\le \var_{1}f,$$
and that the constant in front of $\var_{1}f$ ($1$ in this case) is sharp. Also, in this setting, Temur \cite{Te} (inspired by the beautiful ideas of Kurka \cite{Ku}) concluded that
$$\var_{1}M_{\mathbb{Z}}f\le C \var_{1}f,$$ for a constant $C>10^8.$ It is still an important open problem to find the optimal constant $C$ such that this inequality holds. In the following we use the notation $\var_{1}=:\var.$ 

Maximal functions on finite graphs were studied by Soria and Tradacete in \cite{ST}, where the star and complete graph were of special interest. There, sharp $l^p-$bounds for maximal operators on finite graphs were first obtained. Later, some other geometric properties of maximal functions on infinite graphs were studied by those authors in \cite{ST2}. More recently, bounds for the $p-$variation of the maximal functions on finite graphs were established by Liu and Xue in \cite{LX}. Finding optimal bounds for both the $l^p-$norm of the maximal functions and the $p$-variation of the maximal functions acting on finite graphs is a very interesting and challenging problem. In this paper we make progress on this kind of problem.

\subsection{Conjectures and results for the $p$-variation in finite graphs.}
For a given graph $G=(V,E)$ and $0<p<\infty$, we define 
$${\bf C}_{G,p}:=\sup_{f:V\to \mathbb{R};\var_{p}f>0}\frac{\var M_{G}f}{\var_{p}f}.$$
Liu and Xue (\cite{LX}) obtained optimal results for $n=3$ and for the general case $n>3$ they found some bounds and posed some interesting conjectures. More precisely, they proved that if $G$ is the complete graph with $n$ vertices $K_n$ or the star graph with $n$ vertices $S_n$, then
$$
1-\frac{1}{n}\leq {\bf C}_{G,p}\leq 1
$$
for $0<p<\infty$, and for $n=3$ the lower bound becomes an equality. Moreover, Liu and Xue posed the following conjectures \cite[Conjecture 1]{LX}.

{\it{Conjecture A (for the complete graph $K_n$): For every $n\ge 2$ and $p\in (0,\infty)$ we have
$$
{\bf C}_{K_n,p}=1-\frac{1}{n}.
$$
}}
In this paper we give a positive answer to this conjecture for all $p\ge \frac{\log 4}{\log 6}\approx 0.77$. This range is certainly not optimal and is an interesting problem to try to extend it. Also, we prove the conjecture for every $0<p<1$ when $n=4.$ That is the content of our Theorem \ref{theo 1}. 
\begin{theorem}[Complete graph]\label{theo 1}
Let $0<p\leq\infty$ and $K_{n}=(V,E)$ be a complete graph with $n$ vertices $(a_{1},a_{2},\dots,a_{n})$. Then 
\begin{itemize}
    \item [(i)]  If $p>1$, then \begin{equation*} 
{\bf C}_{K_n,p}=1-\frac{1}{n}.
\end{equation*} 

\item [(ii)] If $0<p\le 1$ and $n=4$, \begin{equation*}
{\bf C}_{K_n,p}=1-\frac{1}{n}.
\end{equation*} 
\item [(iii)] If $n\ge 3$ and $1\ge p\ge \frac{\log(4)}{\log(6)}\approx 0.77$, then
\begin{equation*} 
{\bf C}_{K_n,p}=1-\frac{1}{n}.
\end{equation*} 
\end{itemize}
Moreover, in all the cases the function $\delta_{a_2}$ is an extremizer.
\end{theorem}
We notice that given the different behavior of the function $x\to x^p$ when $p>1$ and $p\le 1$ very contrasting techniques are needed in each case. Also, we observe that proving \eqref{grmaintool} in a larger range implies a proof of Theorem \ref{theo 1} (iii) in the same range. This is the case because the remaining of the proof is independent of the condition $p\ge \frac{\log(4)}{\log(6)}.$
\\
The second conjecture that they posed is the following.

{\it{Conjecture B (for the star graph $S_n$): For any $n\ge 2$ and $p\in (0,1]$ we have  
$$
{\bf C}_{S_n,p}=1-\frac{1}{n}.
$$
}} 
In this case we prove that, in fact, this equality  is not true for $p>1$. In fact, for $n=3$, we find some bounds different to the ones conjectured in that case. However, we give a positive answer to this conjecture when $1/2\leq p\leq1$ for all $n\geq 2$. Moreover, we give a positive answer to the conjecture when $0<p<1/2$ if $n$ is sufficiently large, this is the content of our Theorem \ref{theo 2}.
 \begin{theorem}[Star graph]\label{theo 2}
 Let $S_{n}=(V,E)$ be a start graph with $n$ vertices $(a_{1},a _{2},\dots,a_n)$, with center at $a_1$. Then, the following hold.
\begin{itemize}
\item [(i)] For all $1<p<\infty$ we have that
\begin{align}\label{star graph n=3}
{\bf C}_{S_3,p}=\frac{(1+2^{p/(p-1)})^{(p-1)/p}}{3}<1.
\end{align} 
\item[(ii)] If $p=1$, then  
\begin{equation}\label{star graph p=1}
{\bf C}_{S_n,p}=1-\frac{1}{n}.
\end{equation} 
\item[(iii)] If $n=4$ and $0<p<1,$ or $n\ge 5$ and $\frac{1}{2}\le p\le 1,$ then  \begin{equation}\label{star graph p<1}
{\bf C}_{S_n,p}=1-\frac{1}{n}.
\end{equation} 
 Moreover, \eqref{star graph p<1} holds for every $\frac{1}{2}>p>0$ when $n\ge C(p),$ for some finite constant $C(p)$ depending only on $p.$
\end{itemize}
\end{theorem}
The range $(\frac{1}{2},1)$ in (iii) is certainly not optimal, to find improvements on this range is an interesting problem.  
\\
{\it{Conjecture C (boundedness and continuity): Let $0<p,q\le \infty $ and $0\le \alpha<1.$ The operator $M_{\alpha,G}$ is bounded and continuous from $BV_{p}(G)$ to $BV_{q}(G)$, where $BV_{p}(G):=\{f:V\to \mathbb{R}; \var_pf<\infty\}$ is endowed with $\|f\|_{\widetilde {BV_{p}(V)}}:=\var_{p}f$,  note that $\|\cdot\|_{\widetilde{BV_{p}(V)}}$ depends strongly on $G.$ not only on the set of vertices $V$.}}

We prove that the boundedness holds as conjectured. Moreover,
we prove that with a slight modification the continuity affirmation is true. That is the content of our next theorem.
We also prove that a modification is strictly required. This is related with the fact that $\|\cdot\|_{\widetilde{BV_p(G)}}$ is not a norm (think about constant functions for example), on the other hand, taking $a_0\in V$ we have that $\|f\|_{BV_p(G)}:=\|f\|_{\widetilde{BV_p(G)}}+|f(a_0)|$ is a norm.
\begin{theorem}\label{theo 3}
 Let $G_{n}=(V,E)$ be a graph with $n$ vertices $(a_{1},a _{2},\dots,a_n)$. The following statements hold.
\begin{itemize}
\item[(i)]{\it{[Boundedness]}} Let $\alpha\in[0,1)$. For all $0<p,q\leq\infty$ there exists a constant $C(n,p,q)>0$ such that
\begin{equation}\label{boundedness}
\var_q M_{\alpha,G_{n}}f\leq C(n,p,q)\var_p f.
\end{equation} 
for all functions $f:V\to\R$. 
\item[(ii)]{\it{[Continuity]}} Let $0<p,q\leq\infty$. Consider a sequence of functions 
$f_j:V\to\R$ such that $\|f_j-f_0\|_{{\text{BV}}_p(G)}\to0$ as $j\to\infty$. 
\begin{enumerate}
\item Assuming that $\lim_{j\to\infty}\min_{x\in V}|f(x)-f_j(x)|=0$. Then
\begin{equation}\label{continuity}
\var_q (M_{\alpha,G_{n}}f-M_{\alpha,G_n}f_j)\to0 \ \text{as}\ j\to\infty.
\end{equation}
\item \eqref{continuity} could fail to be true without the extra assumption that $\lim_{j\to \infty} \min_{x\in V}|f(x)-f_j(x)|=0$.
\end{enumerate}
\item[(iii)]  $M_{\alpha,G_n}$ is bounded and continuous from $(BV_p(G_n),\|\cdot\|_{BV_p(G_n)})$ to $(BV_p(G_n),\|\cdot\|_{BV_p(G_n)})$.
\end{itemize}
\end{theorem}\

\subsection{Optimal $l^2$ bounds for maximal operators on finite graphs.}
We are also interested in the $l^p$ norm of $M_G$ when acting on finite graphs. That is, to find the exact value of the expression
$$\sup_{f:V\to \mathbb{R}, f\neq 0}\frac{\|M_{G}f\|_{p}}{\|f\|_{p}}=:\|M_{G}\|_{p}, $$
where $\|g\|_{p}:=\left(\displaystyle\sum_{e\in V}|g(e)|^p\right)^{\frac{1}{p}},$ for $g:V\to \mathbb{R}.$

These norms were first treated by Soria and Tradacete, who found $\|M_{G}\|_p$ when $G=S_n$ and $G=K_n$, where $p\in (0,1)$ (see \cite[Proposition 2.7]{ST} and \cite[Theorem 3.1]{ST}). Their results rely strongly in Jensen's inequality for the function $x\to x^p$ where $p\le 1$, so those methods are not available when $p>1$. In fact, they claimed that this problem was difficult when $p>1$ (see \cite[Remark 2.8]{ST}).
The following inequality was proved by Soria and Tradacete [See \cite{ST}, Proposition 2.7] 
$$
\left(1+\frac{n-1}{n^2}\right)^{1/2}\leq  \|M_{K_n}\|_2\leq \left(1+\frac{n-1}{n}\right)^{1/2}.
$$
Our next result is a formula for the precise value of $\|M_{K_n}\|_{2}$ for $n\ge 2$. We also find  extremizers for all $n\geq2$. Moreover, we prove that $\|M_{K_{3n}}\|_{2}=\|M_{K_{3}}\|_{2},$ for all $n\geq 2$. We list these results as follows.

\begin{theorem}\label{p=2, complete graph}
Let $K_n=(V,E)$ be the complete graph with $n$ vertices $V=\{a_1,a_2,\dots,a_n\}$. Then we have
\begin{equation*}
   \|M_{K_n}\|_{2}= \max_{k\in\{\lfloor\frac{n}{3}\rfloor,\lceil\frac{n}{3}\rceil\}}\left(1-\frac{k}{2n}+\frac{(4kn-3k^2)^{1/2}}{2n}\right)^{1/2},
\end{equation*}
where $\lfloor x\rfloor:=\max\{k\in\Z; k\leq x\}$ (it is the integer part of $x$) and $\lceil x\rceil:=\min\{k\in\Z;k\geq x\}$.
\end{theorem}
In particular, we have.
\begin{corollary}\label{p=2, n=3m}
If $n=3m$ for some $m\in\N$, then
$$\|M_{K_{3m}}\|_2=\left(\frac{4}{3}\right)^{1/2}.$$ For $n=2$ we have $\|M_{K_2}\|_2=\frac{(3+5^{1/2})^{1/2}}{2}$.
\end{corollary}

Similarly, the following inequality was also proved by Soria and Tradacete [See \cite{ST}, Proposition 3.4] 
$$
\left(1+\frac{n-1}{4}\right)^{1/2}\leq \|M_{S_n}\|_2\leq \left(\frac{n+5}{2}\right)^{1/2}.
$$
Our next result is a formula for the precise value of $\|M_{S_n}\|_2$. Moreover, we find some extremizers.

\begin{theorem}\label{p=2, star graph}
Let $n\geq4$ and $S_n=(V,E)$ be the star graph with $n$ vertices $V=\{a_1,a_2,a_3,\dots,a_{n}\}$ and center at $a_1$. Then, the following holds.
\begin{equation}\label{p=2,star}
   \|M_{S_n}\|_2= \left(1+\frac{n-4}{8}+\frac{(n^2+8n)^{1/2}}{8}\right)^{1/2}.
\end{equation}
\end{theorem}
\begin{remark}
It was observed by Soria and Tradecete that in the case $n=2$ the optimal constant is $\frac{[3+5^{1/2}]^{1/2}}{2}$ [See remark 2.8 in \cite{ST}], this coincides with our formula \eqref{p=2,star}.\\ 
\end{remark}


\section{Proof of optimal bounds for the $p$-variation of maximal functions.}
We start by proving our results on $K_n.$\\
\subsection{Optimal bounds for the $p$-variation on $K_n$: Proof of Theorem \ref{theo 1}}

For every result listed in Theorem \ref{theo 1} we can see that, taking $f=\delta_{a_1}$ in the definition of ${\bf C}_{K_n,p},$ we have the following.
$${\bf C}_{K_n,p}\ge 1-\frac{1}{n}.$$
In the following we prove, in each case, that 
\begin{align}\label{grupperbound}
{\bf C}_{K_n,p}\le 1-\frac{1}{n}.
\end{align}

A very important tool in the case $p\le 1$ will be the Karamata's inequality, we include the precise statement of this for completeness:

\begin{lemma}[Karamata's Inequality]\label{Karamata's Inequality}
Let $I$ be an interval of the real line and let $f$ denote a real valued, convex function defined on $I$. If $x_1,\dots,x_n$ and $y_1,\dots,y_n$ are numbers in $I$ such that $(x_1,\dots,x_n)$ majorizes $(y_1,\dots, y_n)$, then
$$
f(x_1)+\dots+f(x_n)\geq f(y_1)+\dots +f(y_n).
$$
Here majorization means that $x_1,\dots, x_n$ and $y_1,\dots,y_n$ satisfies
$$
x_1\geq \dots \geq x_n\ \ \text{and}\ \ y_1\geq \dots\geq y_n,
$$
and we have the inequalities
$$
x_1+x_2+\dots+x_i\geq y_1+y_2+\dots+y_i \ \text{for all}\ \  i\in \{1,\dots,n\},
$$
and the equality
$$
x_1+x_2+\dots+x_n=y_1+y_2+\dots+y_n.
$$
\end{lemma}

\begin{remark}\label{karamata ineq $f(x)=-x^p$}
In the case $0<p\le 1$, in the proof of our Theorems \ref{theo 1} and \ref{theo 2}, we will use Karamata's inequality several times in the particular case when $f(x)=-x^p$. 
\end{remark}

\begin{proof}[Proof of Theorem \ref{theo 1} (i)]

Since by the triangular inequality we have that $\var_p{|f|}\leq \var_pf$ for any function $f:V\to\R$, so we can assume without loss of generality that $f$ is non-negative.
Let $$m:=m_n:=\frac{\sum_{i=1}^{n}f(a_{i})}{n},$$ and for all $k\in\{1,2,\dots,n-1\}$ we define
$$m_k=\frac{\sum_{i=1}^{k}f(a_i)}{k}.$$

Reordering if necessary, we can assume without loss of generality that 
$$
f(a_{n})\geq f(a_{n-1})\geq\dots\geq f(a_{r})\geq m> f(a_{r-1})\geq\dots\geq f(a_{1}),
$$
thus we have that
$$
M_{K_n}f(a_{i})=f(a_{i}) \ \forall\ i\geq r\ \ \  \text{and}\ \ \ M_{K_n}f(a_{i})=m \ \forall\ i<r.
$$
Let us keep in mind in the following that $$(\var_{p}M_{K_n}f)^p=\sum_{i,j\in \{r,\dots,n\}}|f(a_i)-f(a_j)|^p+(r-1)\sum_{i=r}^{n}|f(a_i)-m|^p.$$
Observe that $m_1\leq m_2\leq m_3\leq\dots\leq m_{n-1}\leq m$.
Therefore
\begin{eqnarray}
(\var_p M_{K_n}f)^p&\leq &(n-1)(f(a_{n})-m)^p
+(n-2)(f(a_{n-1})-m)^p\nonumber\\
&&\dots
+(r-1)(f(a_{r})-m)^{p}\nonumber\\
&\leq &(n-1)(f(a_{n})-m)^p
+(n-2)(f(a_{n-1})-m_{n-1})^p\nonumber\\
&&\dots
+(r-1)(f(a_{r})-m_{r})^{p}\label{key ineq 1}.
\end{eqnarray}

Then, we note that by H\"older inequality
\begin{eqnarray*}
f(a_i)-m_i\leq\frac{\sum_{t=1}^{i-1}{|f(a_i)-f(a_t)|}}{i}\leq \frac{\left(\sum_{t=1}^{i}{|f(a_i)-f(a_t)|^p}\right)^{1/p}(i-1)^{1/{p'}}}{i}.
\end{eqnarray*}
where $p'=\frac{p}{p-1}$ denotes the conjugate of $p$ as usual (remind that $p>1$). Combining the two previous estimatives we obtain
\begin{eqnarray*}
(\var_p M_{K_n}f)^p
&\leq &(n-1)(f(a_{n})-m)^p
+(n-2)(f(a_{n-1})-m_{n-1})^p\\
&&\dots
+(r-1)(f(a_{r})-m_{r})^{p}\\
&\leq& (n-1)\frac{\left(\sum_{t=1}^{n-1}{|f(a_n)-f(a_t)|^p}\right)^{p/p}(n-1)^{p/{p'}}}{n^p}\\
&&+ (n-2)\frac{\left(\sum_{t=1}^{n-2}{|f(a_{n-1})-f(a_t)|^p}\right)^{p/p}(n-2)^{p/{p'}}}{(n-1)^p}\\
&&\dots+(r-1)\frac{\left(\sum_{t=1}^{r-1}{|f(a_r)-f(a_t)|^p}\right)^{p/p}(r-1)^{p/{p'}}}{r^p}\\
&\leq& \left(\frac{n-1}{n}\right)^{p}{\sum_{t=1}^{n-1}{|f(a_{n})-f(a_t)|^p}}\\
&&+ \left(\frac{n-2}{n-1}\right)^{p}{\sum_{t=1}^{n-2}{|f(a_{n-1})-f(a_t)|^p}}\\
&&+\dots+ \left(\frac{r-1}{r}\right)^{p}{\sum_{t=1}^{r-1}{|f(a_{r})-f(a_t)|^p}}\\
&\leq& \left(\frac{n-1}{n}\right)^{p}(\var_pf)^p. 
\end{eqnarray*}
From where we conclude \eqref{grupperbound} in this case. Concluding the proof of this assertion of Theorem \ref{theo 1}.

\end{proof}

\subsubsection{Case $p\le 1$: Proof of assertion (ii) and (iii) in Theorem \ref{theo 1}. }
We keep the notation of the previous proof and the assumption that
$$f(a_n)\ge \dots f(a_r)\ge m> \dots f(a_1).$$
For $0<p\le 1$, the simplest case of the theorem  is when $r=n.$
\begin{lemma}
For every $0<p\le 1$ and $n\ge 2$, if  $r=n,$ we have 
${\bf C}_{K_n,p}=1-\frac{1}{n}.$
\end{lemma}
\begin{proof}
This can be proved directly by 
\begin{align*}
(n-1)|f(a_n)-m|^p\le &(n-1)\left|\frac{n-1}{n}(f(a_n)-f(a_1))\right|^p\\
&\le \left(\frac{n-1}{n}\right)^p(|f(a_n)-f(a_1)|^p+\sum_{i=2}^{n-1} (|f(a_n)-f(a_i)|^p+|f(a_i)-f(a_1)|^p)\\
&\le \left(\frac{n-1}{n}\right)^p (\var_p f)^p,
\end{align*}
where, in the second inequality, we used that if $a,b\ge 0,$ then $(a+b)^p\le a^p+b^p.$ 
\end{proof}
Therefore, in the following we assume that $r<n.$
\begin{proof}[Proof of Theorem \ref{theo 1} (ii)]
Now we prove the assertion for $n=4.$ \smallskip
Since the case $r=4$ was already solved, we have two cases left. 
First we treat the case $r=3.$
\smallskip

{\it Case $r=3$. }We have the following inequality.
\begin{align}\label{grline}
    \left(\frac{3}{4}\right)^p(|f(a_4)-f(a_3)|^p+|f(a_3)-f(a_2)|^p+|f(a_2)-f(a_1)|^p)\ge |f(a_4)-f(a_3)|^p+|f(a_3)-m|^p.
\end{align}\\

{\it Step 1: Proving \eqref{grline}.} In order to prove this, we write $f(a_3)-f(a_2)=x$ and $f(a_4)-f(a_3)=y,$ then $m=\frac{f(a_1)+3f(a_2)+2x+y}{4}$ and 
\begin{equation}\label{usefule ineq tool 2}
\frac{f(a_1)+3f(a_2)+2x+y}{4}\le f(a_2)+x \implies f(a_1)+y\le f(a_2)+2x,
\end{equation} 
also 
\begin{equation}\label{useful ineq tool 1}
m\ge f(a_2)\implies f(a_2)\le f(a_1)+2x+y.
\end{equation}
Then
\begin{align*}
    \left(\frac{3}{4}\right)^p(|f(a_4)-f(a_3)|^p+|f(a_3)-f(a_2)|^p+|f(a_2)-f(a_1)|^p)= \left(\frac{3}{4}\right)^p(y^p+x^p+(f(a_2)-f(a_1))^p),
\end{align*}
Consider first the case where $f(a_2)-f(a_1)+2x\le 4y.$ Here, we observe that $f(a_2)-f(a_1)+2x\in [f(a_2)-f(a_1)+2x-y,4y],$ and since $(f(a_2)-f(a_1)+2x)+3y=(f(a_2)-f(a_1)+2x-y)+4y$ we have that $3y\in [f(a_2)-f(a_1)+2x-y,4y],$ and then, by Karamata's inequality, we have $$(3y)^p+(f(a_2)-f(a_1)+2x)^p\ge (4y)^p+(f(a_2)-f(a_1)+2x-y)^p.$$ Now, since  $(3x)^p+(3(f(a_2)-f(a_1)))^p\ge (f(a_2)-f(a_1)+2x)^p,$ we obtain
\begin{align}\label{grline2}
(3y)^p+(3x)^p+(3(f(a_2)-f(a_1)))^p\ge (4y)^p+(f(a_2)-f(a_1)+2x-y)^p,
\end{align}
from where \eqref{grline} follows by observing that $4(f(a_3)-m)=4(f(a_3)-\frac{f(a_1)+3f(a_2)+2x+y}{4})=f(a_2)-f(a_1)-2x-y\le f(a_2)-f(a_1)+2x-y$.\

Now we consider the other case, where $f(a_2)-f(a_1)+2x\ge 4y.$ We do some previous considerations. First, we have that $3^p\ge \frac{4^p-3^p}{4^p}+1$ and $3^p\ge \frac{4^p-3^p}{2^p}+2^p,$ both consequences of the following application of the AM-GM inequality  $$12^p+3^p>6^p+3^p\ge 2(18)^{p/2}> 2(4)^p.$$
Also, let us observe that \begin{align}\label{gr1234}
\begin{split}
(4^p-3^p)\left(\frac{f(a_2)-f(a_1)}{4}+\frac{x}{2}\right)^p&\le (4^p-3^p)\left(\frac{f(a_2)-f(a_1)}{4}\right)^p+(4^p-3^p)\left(\frac{x}{2}\right)^p\\
&=\frac{4^p-3^p}{4^p}(f(a_2)-f(a_1))^p+\frac{4^p-3^p}{2^p}x^p
\end{split}
\end{align}
and \begin{align}\label{gr12345}
(f(a_2)-f(a_1)+2x)^p\le (f(a_2)-f(a_1))^p+2^px^p.
\end{align}
Now, by considering that (here we use $f(a_2)-f(a_1)+2x\ge 4y$)
$$(4^p-3^p)\left(\frac{f(a_2)-f(a_1)}{4}+\frac{x}{2}\right)^p+(f(a_2)-f(a_1)+2x)^p\ge (4^p-3^p)(y)^p+(f(a_2)-f(a_1)+2x-y)^p,$$
we have that by \eqref{gr1234} and \eqref{gr12345} (and the already mentioned inequalities for $3^p$): 
\begin{align*}
(3(f(a_2)-f(a_1)))^p+(3x)^p&\ge \left(\frac{4^p-3^p}{4^p}+1\right)(f(a_2)-f(a_1))^p+\left(\frac{4^p-3^p}{2^p}+2^p\right)x^p\\
&\ge(4^p-3^p)\left(\frac{f(a_2)-f(a_1)}{4}+\frac{x}{2}\right)^p+(f(a_2)-f(a_1)+2x)^p\\
&\ge (4^p-3^p)y^p+(f(a_2)-f(a_1)+2x-y)^p.
\end{align*}
From this \eqref{grline2} follows, and therefore we conclude Step 1. 
\\
In the following, we also need the inequality 
\begin{align}\label{gr2}
\left(\frac{3}{4}\right)^p(|f(a_4)-f(a_2)|^p+|f(a_3)-f(a_1)|^p)\ge |f(a_4)-m|^p+|f(a_3)-m|^p.
\end{align}\\

{\it Step 2: Proving \eqref{gr2}.} We have that \eqref{gr2} is equivalent to
\begin{align*}
(3x+3y)^p+(3x+3(f(a_2)-f(a_1)))^p\ge (f(a_2)-f(a_1)+2x+3y)^p+(f(a_2)-f(a_1)+2x-y)^p,    
\end{align*}
since $4(f(a_4)-m)=4(f(a_2)+x+y-\frac{f(a_1)+3f(a_2)+2x+y}{4})=f(a_2)-f(a_1)+2x+3y$ and $4(f(a_3)-m)=4(f(a_4)-m)-4y=f(a_2)-f(a_1)+2x-y.$ 
Here we distinguish among two cases, the first when $x+4y\ge f(a_2)-f(a_1).$ Here, by the concavity of the function $x\to x^p$, since $$4x+2(f(a_2)-f(a_1))+4y\ge 3x+3y\ge 2x+f(a_2)-f(a_1)-y,$$
and
$$4x+2(f(a_2)-f(a_1))+4y\ge 3x+3(f(a_2)-f(a_1))\ge 2x+f(a_2)-f(a_1)-y,$$
by Karamata's inequality we have \begin{align*}
(3x+3y)^p+(3x+3(f(a_2)-f(a_1)))^p&\ge (4x+2(f(a_2)-f(a_1))+4y)^p+(2x+(f(a_2)-f(a_1))-y)^p\\
&\ge ((f(a_2)-f(a_1))+2x+3y)^p+(f(a_2)-f(a_1)+2x-y)^p, 
\end{align*}
from where \eqref{gr2} follows. 

Now we deal with the the other case, where $x+4y\le f(a_2)-f(a_1).$ We can prove that (this is independent to $x+4y\le f(a_2)-f(a_1)$):
 \begin{align}\label{gr3}
 (f(a_2)-f(a_1)+2x+2y)^p+(f(a_2)-f(a_1)+2x)^p\ge (f(a_2)-f(a_1)+2x+3y)^p+(f(a_2)-f(a_1)+2x-y)^p,
 \end{align} by Karamata's inequality. Also, since $y\le \frac{f(a_2)-f(a_1)}{4}$ and $2x+y\ge f(a_2)-f(a_1)$ (remind 
 \eqref{useful ineq tool 1}), we have $x\ge \frac{3(f(a_2)-f(a_1))}{8},$ therefore we obtain (by just expanding)
 $$(3x)(3x+3(f(a_2)-f(a_1)))\ge \left(\frac{3}{2}(f(a_2)-f(a_1))+2x\right)(f(a_2)-f(a_1)+2x)$$
 and, as a consequence,
 \begin{align}\label{grlnineq}
 \log(3x)+\log(3x+3(f(a_2)-f(a_1)))\ge \log\left(\frac{3}{2}(f(a_2)-f(a_1))+2x\right)+\log(f(a_2)-f(a_1)+2x).
 \end{align}
 Let us observe that, since $x\le x+4y\le f(a_2)-f(a_1),$ we have  \begin{align}\label{grorder}
\log(3x)\le \log(f(a_2)-f(a_1)+2x)\le \log\left(\frac{3}{2}(f(a_2)-f(a_1)+2x\right),
\end{align} let us take then $v:=\log(f(a_2)-f(a_1)+2x)+\log(\frac{3}{2}(f(a_2)-f(a_1)+2x)-\log(3x),$ by \eqref{grlnineq} we have $v\le \log(3x+3(f(a_2)-f(a_1)))$ and by \eqref{grorder} we have $v\ge \log\left(\frac{3}{2}(f(a_2)-f(a_1))+2x\right).$ By Karamata's inequality, now applied to the convex function $x\to e^{px}$, by considering \eqref{grorder} we have
 \begin{align*}
  e^{p\log (\frac{3}{2}(f(a_2)-f(a_1))+2x)}+e^{p\log(f(a_2)-f(a_1)+2x)}\le e^{p\log(3x)}+e^{pv}\le e^{p\log(3x)}+e^{p\log(3x+3(f(a_2)-f(a_1)))},   
 \end{align*}
 and therefore:
 \begin{align*}
(3x+3y)^p+(3x+3(f(a_2)-f(a_1)))^p &\ge (3x)^p+(3x+3(f(a_2)-f(a_1)))^p=e^{p\log(3x)}+e^{p\log(3x+3(f(a_2)-f(a_1)))}\\
&\ge e^{p\log(\frac{3}{2}(f(a_2)-f(a_1))+2x)}+e^{p\log(f(a_2)-f(a_1)+2x)}\\
&=(\frac{3}{2}(f(a_2)-f(a_1))+2x)^p+(f(a_2)-f(a_1)+2x)^p\\ 
&\ge  (f(a_2)-f(a_1)+2x+2y)^p+(f(a_2)-f(a_1)+2x)^p,
 \end{align*}
where we used $f(a_2)-f(a_1)\ge 4y$ in the last inequality. Therefore we obtain, by combining this with \eqref{gr3}, the desired inequality \eqref{gr2}.\\ 

{\it Step 3: Conclusion of case $r=3$.} The case $r=3$ then follows by combining \eqref{gr2}, \eqref{grline} and the inequality $\left(\frac{3}{4}\right)^p(f(a_4)-f(a_1))^p= \left(f(a_4)-\frac{f(a_4)+3f(a_1)}{4}\right)^p\ge (f(a_4)-m)^p.$ In fact, adding these three inequalities we obtain that
\begin{align*}
\left(\frac{3}{4}\right)^p(\var_{p}f)^p&=\left(\frac{3}{4}\right)^p\left(|f(a_4)-f(a_3)|^p+|f(a_3)-f(a_2)|^p+|f(a_2)-f(a_1)|^p\right)\\
&\ \ \ \ \ +\left(\frac{3}{4}\right)^p\left(|f(a_4)-f(a_2)|^p+|f(a_3)-f(a_1)|^p\right)+\left(\frac{3}{4}\right)^p|f(a_4)-f(a_1)|^p\\
&\geq \left(|f(a_4)-f(a_3)|^p+|f(a_3)-m|^p\right)+\left(|f(a_4)-m|^p+|f(a_3)-m|^p\right)+|f(a_4)-m|^p\\
&=(\var_{p}M_{K_4}f)^{p}.
\end{align*}
\smallskip{}

{\it Case $r=2$.} Here, we have that $m\le f(a_2)\implies 2x+y\le f(a_2)-f(a_1)$ since $f(a_2)-m=f(a_2)-\frac{f(a_1)+3f(a_2)+2x+y}{4}=\frac{f(a_2)-f(a_1)-2x-y}{4}.$ We prove first the inequality
\begin{align}\label{grliner=2}
|f(a_4)-f(a_3)|^p+|f(a_3)-f(a_2)|^p+|f(a_2)-m|^p&\le \left(\frac{3}{4}\right)^p(|f(a_4)-f(a_3)|^p\\
&+|f(a_3)-f(a_2)|^p+|f(a_2)-f(a_1)|^p)\nonumber,
\end{align}\\

{\it Step 1: Proving \eqref{grliner=2}.} Our desired inequality \eqref{grliner=2} is equivalent (since $4(f(a_2)-m)=4(f(a_2)-\frac{f(a_1)+3f(a_2)+2x+y}{4})=f(a_2)-f(a_1)-2x-y$) to
$$(4y)^p+(4x)^p+(f(a_2)-f(a_1)-2x-y)^p\le (3y)^p+(3x)^p+(3(f(a_2)-f(a_1)))^p.$$
We observe that $$(4y)^p+(4x)^p+(f(a_2)-f(a_1)-2x-y)^p\le (4y)^p+(4x)^p+(f(a_2)-f(a_1))^p,$$
also, since $x\le \frac{f(a_2)-f(a_1)}{2}$ and $y\le f(a_2)-f(a_1)$, we have $$(4^p-3^p)(x^p+y^p)\le (4^p-3^p)\left[\left(\frac{1}{2}\right)^p+1\right](f(a_2)-f(a_1))^p\le (3^p-1)(f(a_2)-f(a_1))^p,$$ because $4^p+8^p+2^p\le 2(6)^p+3^p$ by Jensen inequality. Therefore  $$(4y)^p+(4x)^p+(f(a_2)-f(a_1)-2x-y)^p\le 4^px^p+4^py^p+(f(a_2)-f(a_1))^p\le (3y)^p+(3x)^p+(3(f(a_2)-f(a_1)))^p,$$
from where it follows \eqref{grliner=2}. This conclude Step 1.

Also, we have that 
\begin{align}\label{grlast}
(|f(a_4)-f(a_2)|^p+|f(a_3)-f(a_1)|^p)\left(\frac{3}{4}\right)^p\ge (|f(a_4)-f(a_2)|^p+|f(a_3)-m|^p),
\end{align}
{\it Step 2: Proving \eqref{grlast}.} We have that \eqref{grlast} is equivalent to $$(4x+4y)^p+(f(a_2)-f(a_1)+2x-y)^p\le (3x+3y)^p+(3(x+f(a_2)-f(a_1)))^p,$$
this happens since $f(a_4)-f(a_2)=x+y,$ $f(a_3)-f(a_1)=x+f(a_2)-f(a_1)$ and $f(a_3)-m=f(a_2)+x-\frac{f(a_1)+3f(a_2)+2x+y}{4}=\frac{f(a_2)-f(a_1)+2x-y}{4}.$
Also, since by Jensen  $4^p+2^p\le 2(3)^p$ and because of $m\le f(a_2)\implies 2x+y\le f(a_2)-f(a_1)\implies x+y\le f(a_2)-f(a_1)$ , we have

\begin{align*}
(4^p-3^p)(x+y)^p+(2(f(a_2)-f(a_1)))^p&\le (3(f(a_2)-f(a_1)))^p\le (3(f(a_3)-f(a_1)))^p\\
&=(3(x+f(a_2)-f(a_1)))^p.     
\end{align*}
By observing that $2x\le f(a_2)-f(a_1)$ and then $f(a_2)-f(a_1)+2x-y\le 2(f(a_2)-f(a_1)$ we have
\begin{align*}
(4^p-3^p)(x+y)^p+(f(a_2)-f(a_1)+2x-y)^p&\le (4^p-3^p)(x+y)^p+(2(f(a_2)-f(a_1)))^p\\
&\le (3(x+f(a_2)-f(a_1)))^p,
\end{align*}
from where we obtain \eqref{grlast} and therefore we conclude Step 2.

{\it Step 3: Conclusion of $r=2$ and $n=4.$} The case $r=2,$ and thus our result in $n=4$ follows by combining \eqref{grliner=2}, \eqref{grlast} and the inequality $f(a_4)-m\le (\frac{3}{4})(f(a_4)-f(a_1)).$ We conclude this part and thus the assertion.
\end{proof}

Now we prove our assertion for general $n$ and $p\in [\frac{\log(4)}{\log(6)},1].$\\
\begin{proof}[Proof of Theorem \ref{theo 1} (iii)]
The strategy that we follow in order to prove this assertion is the following inductive argument. To prove \eqref{grupperbound} is equivalent to prove that for each $f:V\to \mathbb{R}_{\ge 0},$ we have 
\begin{align}\label{grupperbound2}
\sum_{i,j\in \{r,\dots,n\}}|f(a_i)-f(a_j)|^p+\sum_{i=r}^{n}|f(a_i)-m|^p\le \left(1-\frac{1}{n}\right)^p(\sum_{i,j\in \{1,\dots,n\}}|f(a_i)-f(a_j)|^p).
\end{align}
In order to prove the inequality above  for $\frac{\log(4)}{\log(6)}\le p\le 1,$ we establish in this range a control over the contribution of the vertex $a_r$ to \eqref{grupperbound2}, that is, we prove that 
\begin{align}\label{grmaintool0}
    \sum_{i=r+1}^n|f(a_i)-f(a_r)|+(n-r)|f(a_r)-m|^p\le \left(1-\frac{1}{n}\right)^p(\sum_{i=1}^{n}|f(a_r)-f(a_i)|^p.
\end{align}
Then, we observe that the analogous of \eqref{grupperbound2} for the graph $K_{n-1},$  obtained by deleting the vertex $a_r$ (and the respective edges) from $K_n,$ and a good choice of $\widetilde{f}:V\setminus \{a_r\}$ would give us a proper control over the edges of $K_n$ that were not considered in \eqref{grmaintool0}. By combining this control with \eqref{grmaintool0} we would obtain \eqref{grupperbound2} for our initial $f$.  Thus, the induction concludes our result (since we know that the result holds for $n=3$).

We proceed to our proof by proving first \eqref{grmaintool0}.
 We write $x_i:=f(a_i)-f(a_r)$ for $i=n,\dots,r+1$, $u=f(a_r)-m$, $y_i=f(a_r)-f(a_i)$ for $i=r-1,\dots,1$. We have then, since $$m=\frac{\sum_{i=1}^nf(a_i)}{n}=\frac{\sum_{i=1}^{r-1}f(a_r)-y_i+f(a_r)+\sum_{i=r+1}^{n}f(a_r)+x_i}{n}=f(a_r)+\frac{\sum_{i=r+1}^{n}x_i-\sum_{i=1}^{r}y_i}{n},$$ that  $\displaystyle \sum_{i=r+1}^n x_i+nu=\displaystyle \sum_{i=1}^{r-1} y_i.$ 
Then, \eqref{grmaintool0} is equivalent to:

\begin{align}\label{grmaintool}
\sum_{i=r+1}^{n}x_i^p+(r-1)u^p\le \left(1-\frac{1}{n}\right)^p\left(\sum_{i=r+1}^{n}x_i^p+\sum_{i=1}^{r-1}y_i^p\right).
\end{align}
In order to prove that let us see that, since $y_i\ge u$ for every $i=1,\dots,r-1$ and $\displaystyle\sum_{i=1}^{r-1}y_i=(r-2)u+\left(\displaystyle\sum_{i=1}^{r-1}y_i-(r-2)u\right)$, we have $\displaystyle\sum_{i=1}^{r-1}y_i-(r-2)u\ge y_j\ge u$ for every $j=1,\dots,r-1.$ Moreover, $y_1\ge y_2\ge \dots \ge y_{r-1}$ and  $\left(\displaystyle\sum_{i=1}^{r-1}y_i-(r-2)u\right)+ku\ge \displaystyle\sum_{i=1}^{k+1}y_i$ for each $k=0,\dots r-2$ since $\displaystyle\sum_{i=k+2}^{r-1}y_i\ge (r-2-k)u,$ then by Karamata's inequality, we have (we also use here that $\displaystyle\sum_{i=r+1}^{n}x_i+nu=\displaystyle\sum_{i=1}^{r-1}y_i$) $$\sum_{i=1}^{r-1}y_i^p\ge (r-2)u^p+\left(\sum_{i=1}^{r-1}y_i-(r-2)u\right)^p=(r-2)u^p+\left((n-r+2)u+\sum_{i=r+1}^{n}x_i\right)^p,$$
also, by Jensen's inequality we have $\left((n-r+2)u+\displaystyle\sum_{i=r+1}^{n}x_i\right)^p\ge 2^{p-1}\left[((n-r+2)u)^p+\left(\displaystyle\sum_{i=r+1}^nx_i\right)^p\right].$ Therefore, combining the two previous inequalities we obtain
\begin{align*}
 \left(1-\frac{1}{n}\right)^p\left(\sum_{i=r+1}^{n}x_i^p+\sum_{i=1}^{r-1}y_i^p\right)\ge \left(1-\frac{1}{n}\right)^p\left[\sum_{i=r+1}^nx_i^p+(r-2)u^p+2^{p-1}((n-r+2)u)^p+2^{p-1}\left(\sum_{i=r+1}^nx_i\right)^p\right].
\end{align*}
Then, since $\left(\displaystyle\sum_{i=r+1}^{n} x_i\right)^p\ge (n-r)^{p-1}\left(\displaystyle\sum_{i=r+1}^{n} x_i^p\right)$ (by H\"older Inequality) we get 
\begin{align*}
 \left(1-\frac{1}{n}\right)^p\left(\sum_{i=r+1}^{n}x_i^p+\sum_{i=1}^{r-1}y_i^p\right)\ge \left(1-\frac{1}{n}\right)^p\left(\sum_{i=r+1}^nx_i^p(1+2^{p-1}(n-r)^{p-1})+u^p(r-2+2^{p-1}(n-r+2)^p)\right).
\end{align*}
Then, in order to obtain \eqref{grmaintool} is enough 
\begin{align}\label{grmaintoolpart1}
1\le \left(1-\frac{1}{n}\right)^p\left(1+2^{p-1}(n-r)^{p-1}\right)
\end{align}
and 
\begin{align}\label{grmaintoolpart2}
r-1\le \left(1-\frac{1}{n}\right)^p\left(r-2+2^{p-1}(n-r+2)^p\right).
\end{align}
To prove \eqref{grmaintoolpart1} is enough (since $r\le n-1\implies n-r\ge 1$ and $n\ge 3\implies 1-\frac{1}{n}\ge \frac{2}{3}$)
\begin{align*}
1\le \left(\frac{2}{3}\right)^p(1+2^{p-1})    
\end{align*}
and that is equivalent to $\left(\frac{3}{2}\right)^p\le 1+2^{p-1},$ which follows since $p\in(0,1)$ and then $(\frac{3}{2})^p\leq \frac{3}{2}=1+\frac{1}{2}\leq 1+\frac{2^p}{2}$. 
Now, to prove \eqref{grmaintoolpart2}, we observe that (since $p\ge \frac{\log(4)}{\log(6)}$, $n-r+2\ge 3$),  $$2^{p-1}(n-r+2)^p\ge 2^{p-1}(3)^p=\frac{6^p}{2}\ge 2,$$ and therefore 
$$\left(1-\frac{1}{n}\right)^p(r-2+2^{p-1}(n-r+2)^p)\ge \left(1-\frac{1}{n}\right)^p(r)\ge \left(1-\frac{1}{n}\right)r\ge (r-1),$$
where in the last inequality we use $1-\frac{1}{n}\ge 1-\frac{1}{r}$ since $r\le n.$ From this we conclude \eqref{grmaintoolpart2}  and thus \eqref{grmaintool} follows. 

\

Now we follow with the remaining steps of our proof. Assume that inequality $\var_{p}M_{K_{n-1}}f\le (1-\frac{1}{n-1})\var_{p}f,$ 
holds for every $\widetilde{f}:V(K_{n-1})\to \mathbb{R}_{\ge 0}$ in $K_{n-1},$ whenever $p\ge \frac{\log(4)}{\log(6)}$ (it holds for $n=3,4$). Then, if $b_1,\dots,b_{n-1}$ are the vertex of the $K_{n-1}$ graph, we define $\widetilde{f}$ as $\widetilde{f}(b_i)=f(a_{i+1})$ for $i=r,\dots,n-1$ and $\widetilde{f}(b_i)=f(a_{i})$ for $i=1,\dots,r-1.$ We write $\widetilde{m}=\frac{\sum_{i=1}^{n-1}\widetilde{f}(b_i)}{n-1}.$ Since $f(a_r)\geq m$, we observe that $\widetilde m=\frac{nm-f(a_r)}{n-1}\leq m$. Then, we write $\widetilde{f}(b_{s})\ge \widetilde{m}>\widetilde{f}(b_{s-1})$, where we observe that $s\le r.$ By the inductive hypothesis, we have
\begin{align}\label{induction1gr}
\sum_{i,j\in \{s,\dots,n-1\}}|\widetilde{f}(a_i)-\widetilde{f}(a_j)|^p+(s-1)\sum_{i=s}^{n-1}|\widetilde{f}(a_i)-\widetilde{m}|^p\le \left(1-\frac{1}{n-1}\right)^p\left(\sum_{i,j\in \{1,\dots,n-1\}}|\widetilde{f}(a_i)-\widetilde{f}(a_j)|^p\right).
\end{align}
By noticing that
\begin{align}\label{induction2gr}
\begin{split}
\sum_{i=s}^{n-1}|\widetilde{f}(a_i)-\widetilde{f}(a_j)|^p&= \sum_{i,j\in \{r+1,\dots,n\}}|f(a_i)-f(a_j)|^p+\sum_{i=r+1}^{n}\sum_{j=s}^{r-1}|f(a_i)-f(a_j)|^p+\sum_{i,j\in \{s,\dots,r-1\}}|f(a_i)-f(a_j)|^p\\
&\ge \sum_{i,j\in \{r+1,\dots,n\}}|f(a_i)-f(a_j)|^p+(r-s)\sum_{i=r+1}^{n}|f(a_i)-m|^p
\end{split}
\end{align}
where in the last inequality we used that $f(a_i)\ge m>f(a_j)$ for $i\ge r+1$ and $j\le r-1.$ Then, by combining \eqref{induction1gr}, \eqref{induction2gr} and using that $(s-1)\displaystyle\sum_{i=s}^{n-1}|\widetilde{f}(a_i)-\widetilde{m}|^p\ge (s-1)\displaystyle\sum_{i=r+1}^n|f(a_i)-m|^p$ (since $\widetilde{m}\le m$), we have 
\begin{align*}
 &\left(1-\frac{1}{n-1}\right)^p\left(\sum_{i,j\in \{1,..r-1,r+1,..,n\}}|f(a_i)-f(a_j)|^p\right)\\
 &\ge \sum_{i,j\in \{r+1,\dots,n\}}|f(a_i)-f(a_j)|^p+(r-s)\sum_{i=r+1}^{n}|f(a_i)-m|^p\\
 &\ \ \ \ +(s-1)\sum_{i=r+1}^{n}|f(a_i)-m|^p\\
 &\ge \sum_{i,j\in \{r+1,\dots,n\}}|f(a_i)-f(a_j)|^p+(r-1)\sum_{i=r+1}^{n}|f(a_i)-m|^p.
\end{align*}
Combining this with \eqref{grmaintool0} we conclude $$\sum_{i,j\in \{r,..,n\}}|f(a_i)-f(a_j)|^p+(r-1)\sum_{i=r}^n |f(a_i)-m|^p\le \left(1-\frac{1}{n}\right)^p\left(\sum_{i,j\in \{1..,n\}}|f(a_i)-f(a_j)|^p\right),$$
that is equivalent to \eqref{grupperbound} in this case. This concludes the proof of our theorem.
\end{proof}



\subsection{0ptimal bounds for the $p$-variation on $S_n$: Proof of Theorem \ref{theo 2}}

Now we deal with the problems related to the $p$-variation of the maximal operator in $S_n.$ 

\begin{proof}[Proof of Theorem \ref{theo 2} (i)]
We assume without loss of generality that $f$ is non-negative. We analyse three different cases. {\it{Case 1: 
$f(a_1)\geq\max\{f(a_2),f(a_3)\}$
}}.\\
In this case we have that $M_{S_3}f(a_1)=f(a_1)$, then

\begin{align*}
(\var_p M_{S_{3}}f)^p&\leq \left(f(a_1)-\frac{f(a_1)+f(a_2)}{2}\right)^p+\left(f(a_1)-\frac{f(a_1)+f(a_3)}{2}\right)^p\\
&\leq 
\frac{1}{2^{p}}(\var_p f)^p.
\end{align*}

{\it{Case 2: $f(a_1)\leq\min\{f(a_2),f(a_3)\}$
}}. We assume without loss of generality that 
$f(a_1)\leq f(a_3)\leq f(a_2)$. Then, by H\"older inequality we have that

\begin{align*}
(\var_p M_{S_{3}}f)^p&= \left(f(a_2)-\frac{f(a_1)+f(a_2)+f(a_3)}{3}\right)^p+\left(\left[f(a_3)-\frac{f(a_1)+f(a_2)+f(a_3)}{3}\right]_{+}\right)^p\\
&=\left(\frac{f(a_2)-f(a_1)+f(a_2)-f(a_3)}{3}\right)^p+\left(\left[\frac{f(a_3)-f(a_1)-(f(a_2)-f(a_3))}{3}\right]_{+}\right)^p\\
&=\left(\frac{2(f(a_2)-f(a_1))-(f(a_3)-f(a_1))}{3}\right)^p+\left(\left[\frac{2(f(a_3)-f(a_1))-(f(a_2)-f(a_1))}{3}\right]_{+}\right)^p\\
&\leq \left(\frac{2(f(a_2)-f(a_1))-(f(a_3)-f(a_1))}{3}+\left[\frac{2(f(a_3)-f(a_1))-(f(a_2)-f(a_1))}{3}\right]_{+}\right)^p \\
&\leq \frac{(1+2^{p'})^{p/p'}}{3^p}(\var_p f)^p.
\end{align*}
Where we have used the fact that $p>1$ in the fourth line and the final step follows by H\"older's inequality.\\

{\it{Case 3: $\min\{f(a_2),f(a_3)\}<f(a_1)<\max\{f(a_2),f(a_3)\}$}}. We assume without loss of generality that 
$f(a_3)< f(a_1)< f(a_2)$. Then, we have

\begin{align*}
(\var_p M_{S_{3}}f)^p&= \left(f(a_2)-M_{S_3}f(a_1)\right)^p+\left(M_{S_3}f(a_1)-\frac{f(a_1)+f(a_2)+f(a_3)}{3}\right)^p\\
&\leq \left(f(a_2)-\frac{f(a_1)+f(a_2)+f(a_3)}{3}\right)^p\\
&=\left(\frac{2(f(a_2)-f(a_1))+(f(a_1)-f(a_3))}{3}\right)^p\\
&\leq \frac{(1+2^{p'})^{p/p'}}{3^p}(\var_p f)^p.
\end{align*}
In the second line we used the fact that $p>1$. In the final inequality (in the last two cases) we have used that $c_1d_1+c_2d_2\leq (c^{p'}_1+c^{p'}_2)^{1/p'}(d^p_1+d^p_2)^{1/p}$ for all $c_1,c_2,d_1,d_2\geq0$ by H\"older's inequality.
This conclude the proof of
\begin{align*}
{\bf C}_{S_3,p}\le \frac{(1+2^{p/(p-1)})^{(p-1)/p}}{3}.   \end{align*}

in \eqref{star graph n=3}. Finally, we observe that 
\begin{align}\label{grupperboundn=3}
{\bf C}_{S_3,p}\ge \frac{(1+2^{p/(p-1)})^{(p-1)/p}}{3}.    
\end{align}
For that we consider the function $f:V\to\R$ defined by
$$
f(a_3)=2, f(a_1)=3 \ \text{and}\ f(a_2)=3+2^{\frac{1}{p-1}}.
$$
Then, $\var_pf=(1+2^{\frac{p}{p-1}})^{\frac{1}{p}}$. Moreover,
$$
M_{S_3}f(a_2)=f(a_2)=3+2^{\frac{1}{p-1}}\ \text{and}\ \ M_{S_3}(a_3)=M_{S_3}f(a_1)=\frac{2+3+3+2^{\frac{1}{p-1}}}{3}.
$$
Thus
$$
\var_{p} M_{S_3}f=M_{S_3}f(a_2)-M_{S_3}f(a_1)=\frac{1+2^{\frac{p}{p-1}}}{3}.
$$
Therefore
$$\frac{\var_{p}M_{S_3}f}{\var_{p}f}=\frac{(1+2^{p'})^{\frac{1}{p'}}}{3}.$$
So, we obtain \eqref{grupperboundn=3} and thus \eqref{star graph n=3}.
\end{proof}
The proof of the previous result provides an example where the value $$\sup_{f:V\to \mathbb{R}; \var_{p}f>0}\frac{\var_{p}M_{G}f}{\var_{p}f}$$ is not attained by any {\it Dirac delta.} This is a sign of the complexity of this problem when $p>1,$ since is not clear how the extremizers should behave for $n>3.$  

In the case $p=2$, an interesting example is the following: let $S_n=(V,E)$ as in the Theorem \ref{theo 2}, consider the function $f:V\to\R$ defined by
$$
f(a_1)=n, \ f(a_2)=n+(n-1),\ \  \text{and}\ \ f(a_i)=n-1 \ \ \text{for all}\ \ 3\leq i\leq n.
$$
In this case
$$
M_{S_n}f(a_2)=n+(n-1)\ \ \text{and}\ \    M_{S_n}f(a_i)=n+\frac{1}{n} \ \text{for all}\ \ i\neq 2.
$$
Then
$$
\frac{\var_{2}M_{S_n}f}{\var_2f}=\frac{n-1-\frac{1}{n}}{[(n-1)^2+(n-2)]^{1/2}}=\frac{[(n-1)^2+(n-2)]^{1/2}}{n}>\frac{n-1}{n}.
$$
This provides further evidence to the fact that in general the extremizers on $S_n$ are different when $p>1$ than when $p\le 1.$ \\
Now we deal with the next assertion of our theorem.  
Taking $f=\delta_{a_2}$ on the definition of ${\bf C}_{S_n,p}$ we have that
$${\bf C}_{S_n,p}\ge 1-\frac{1}{n}.$$
In the following we prove the inequality
\begin{align}\label{grupperboundestrella}
{\bf C}_{S_n,p}\le 1-\frac{1}{n},
\end{align}
from where both assertion follow. This inequality is equivalent to
\begin{align}\label{grinequalitystar}
\var_{p} M_{S_{n}}f\le (1-\frac{1}{n}) \var_{p}f,    
\end{align}
for all functions $f:V\to \mathbb{R}.$
\begin{proof}[Proof of Theorem \ref{theo 2} (ii)]
We assume without loss of generality that $f$ is non-negative. Let 
$$
m=\frac{1}{n}\sum_{i=1}^{n}f(a_i).
$$
Then
\begin{align*}
    &\var M_{S_n}f\\
    &=\sum_{i=2}^{n}|M_{S_n}f(a_i)-M_{S_n}f(a_1)|\\
    &=\sum_{M_{S_n}f(a_i)>M_{S_n}f(a_1)}M_{S_n}f(a_i)-M_{S_n}f(a_1)+\sum_{M_{S_n}f(a_1)>M_{S_n}f(a_i)}M_{S_n}f(a_1)-M_{S_n}f(a_i)\\
    &= \sum_{M_{S_n}f(a_i)>M_{S_n}f(a_1)}f(a_i)-M_{S_n}f(a_1)+\sum_{M_{S_n}f(a_1)>M_{S_n}f(a_i)}f(a_1)-M_{S_n}f(a_i)\\
    &\leq \sum_{M_{S_n}f(a_i)>M_{S_n}f(a_1)}f(a_i)-m+\sum_{M_{S_n}f(a_1)>M_{S_n}f(a_i)}f(a_1)-m\\
    &=\sum_{M_{S_n}f(a_i)>M_{S_n}f(a_1)}\left[\frac{n-1}{n}(f(a_i)-f(a_1))+\sum_{j\neq i}\frac{f(a_1)-f(a_j)}{n}\right]\\
    &\ \ \ \ \ \ \ \ \ \ +\sum_{M_{S_n}f(a_1)>M_{S_n}f(a_i)}\sum_{k=2}^{n}\frac{f(a_1)-f(a_k)}{n}\\
    &=\sum_{M_{S_n}f(a_i)>M_{S_n}f(a_1)}(f(a_i)-f(a_1))\left[\frac{n-1}{n}-\frac{(|\{i;M_{S_n}f(a_i)>M_{S_n}f(a_1)\}|-1)}{n}\right. \\
    &\ \ \ \ \ \ \ \ \ \ \ \ \ \ \ \ \ \ \ \ \ \ \ \ \ \ \ \ \ \ \ \ \ \ \ \ \ \ \ \ \ \ \ \ \ \ \ \ \ \ \ \ \ \ \ \ \ \  \left.-\frac{|\{i;M_{S_n}f(a_1)>M_{S_n}f(a_i)\}|}{n}\right]\\
    &\ \ \ \ \ \ \ \ \ \ +\sum_{M_{S_n}f(a_1)>M_{S_n}f(a_i)}(f(a_1)-f(a_k))\left[\frac{|\{i;M_{S_n}f(a_i)>M_{S_n}f(a_1)\}|}{n}\right. \\
    &\ \ \ \ \ \ \ \ \ \ \ \ \ \ \ \ \ \ \ \ \ \ \ \ \ \ \ \ \ \ \ \ \ \ \ \ \ \ \ \ \ \ \ \ \ \ \ \ \ \ \ \ \ \ \ \ \ \  \left.+\frac{|\{i;M_{S_n}f(a_1)>M_{S_n}f(a_i)\}|}{n}\right]\\\\
    &\leq \frac{n-1}{n}\var f, 
\end{align*}
from where \eqref{grupperboundestrella} follows and therefore our result. 
\end{proof}
\begin{proof}[Proof of Theorem \ref{theo 2} (iii)]
We write $f(a_2)\ge \dots \ge f(a_r)\ge m>f(a_{r+1})\ge \dots\ge f(a_{n}).$ We distinguish among two cases, the first being $f(a_1)\le m.$\\

{\it Case 1: $f(a_1)\le m.$} Let us keep in mind in the following that in this case
$$(\var_{p}M_{S_n}f)^p=\sum_{i=2}^{r}|f(a_i)-m|^p.$$

In this case is enough to prove inequality \eqref{grinequalitystar} when $f(a_i)<f(a_1)$ for $i>r$. In fact, if \eqref{grinequalitystar} fails for some $f$ with $f(a_i)>f(a_1)$ and $i>r$, it also fails for the function $\widetilde{f}$ defined by $\widetilde{f}(e)=f(e)$ for every $e\notin \{a_2,a_i\}$, $\widetilde{f}(a_i)=2f(a_1)-f(a_i)$ and $\widetilde{f}(a_2)=f(a_2)+f(a_i)-\widetilde{f}(a_i)$ (notice that $\widetilde{f}(a_i)<\widetilde{f}(a_1)$ by construction). This holds because $\widetilde{m}=\frac{\sum_{j=1}^n\widetilde{f}(a_j)}{n}=\frac{\sum_{i=j}^nf(a_j)}{n}=m$, by definition, and \begin{align}\label{grreduction}
(1-\frac{1}{n})^p(f(a_2)-f(a_1))^p-(f(a_2)-m)^p\ge (1-\frac{1}{n})^p(\widetilde{f}(a_2)-\widetilde{f}(a_1))^p-(\widetilde{f}(a_2)-\widetilde{m})^p,
\end{align}
noticing that\eqref{grreduction} is equivalent to $$\left(1-\frac{1}{n}\right)^p(\var_{p}f)^p-(\var_{p}M_{S_n}f)^p\ge \left(1-\frac{1}{n}\right)^p(\var_{p}\widetilde{f})^p-(\var_{p}M_{S_n}\widetilde{f})^p,$$
since the other terms in this inequality remain unchanged when we do the transformation $f\to \widetilde{f}$ (notice that, by construction, $|f(a_1)-f(a_i)|=|\widetilde{f}(a_1)-\widetilde{f}(a_i)|.$)
We have that \eqref{grreduction} holds because
$$|\widetilde{f}(a_2)-\widetilde{f}(a_1)|^p-|f(a_2)-f(a_1)|^p\le |\widetilde{f}(a_2)-\widetilde{m}|^p-|f(a_2)-m|^p,$$ inequality that follows because of $\widetilde{f}(a_1)=f(a_1), m=\widetilde{m}$, the concavity of the function $x\to x^p$ (and thus the function $x\to (x+c)^p-x^p$ is decreasing for $x,c>0$, here considering $c=\widetilde{f}(a_2)-f(a_2)$)  and the fact that $f(a_2)-f(a_1)\ge f(a_2)-m.$ By iterating the previous argument we get the desired reduction.

We write $f(a_i)-m=x_i$ for $i=2,...r;$ $m-f(a_1)=u$ and $y_i=f(a_1)-f(a_i)$ for $i=r+1,...,n.$ Observe that given our reduction we have $y_i\ge 0.$ We observe that since \begin{align*}m=\frac{\sum_{i=1}^n f(a_i)}{n}&=\frac{\sum_{i=2}^r(m+x_i)+f(a_1)+\sum_{i=r+1}^n(f(a_1)-y_i)}{n}\\
&=\frac{\sum_{i=2}^r(m+x_i)+m-u+\sum_{i=r+1}^{n}(m-u-y_i)}{n},
\end{align*}
we have
\begin{align*}
\sum_{i=2}^r x_i=u+\sum_{i=r+1}^{n}(u+y_i), 
\end{align*}
from where we obtain $u\le \frac{\sum_{i=2}^r x_i}{n-r+1}.$
Also, let us observe that \eqref{grinequalitystar} is equivalent in this case to 
\begin{align}\label{grstar1234}
\sum_{i=2}^{r}|x_i|^p\le \left(1-\frac{1}{n}\right)^p\left(\sum_{i=2}^r|x_i+u|^p+\sum_{i=r+1}^{n}|y_i|^p\right).
\end{align}
Observe that $\displaystyle\sum_{i=r+1}^{n}|y_i|^p\ge \left|\displaystyle\sum_{i=r+1}^{n}y_i\right|^p=\left|\displaystyle\sum_{i=2}^{r} x_i-(n-r+1)u\right|^p.$ Then,for $x_2,\dots, x_r,n,r$ and $p$ fixed, we define the function $$g(z):=\displaystyle\sum_{i=2}^r|x_i+z|^p+\left|\displaystyle\sum_{i=2}^{r} x_i-(n-r+1)z\right|^p,$$ we observe that for $z\in \left[0, \frac{\sum_{i=2}^r x_i}{n-r+1}\right]$ this function is concave (sum of concave functions), therefore $g(z)\ge \min \left\{g\left(\frac{\sum_{i=2}^r x_i}{n-r+1}\right),g(0)\right\}$ in that interval. Then, we have \begin{align*}
\left(1-\frac{1}{n}\right)^p\left(\sum_{i=2}^r|x_i+u|^p+\sum_{i=r+1}^{n}|y_i|^p\right)&\ge \left(1-\frac{1}{n}\right)^p\left(\sum_{i=2}^r|x_i+u|^p+|\displaystyle\sum_{i=2}^{r} x_i-(n-r+1)u|^p\right)\\
&=\left(1-\frac{1}{n}\right)^pg(u)\\
&\ge \left(1-\frac{1}{n}\right)^p\min \left\{g\left(\frac{\sum_{i=2}^r x_i}{n-r+1}\right),g(0)\right\}\\
&\ge \left(1-\frac{1}{n}\right)^p\min\left\{\left(\sum_{i=2}^{r}|x_i|^p+\left|\sum_{i=2}^{r}x_i\right|^p\right),\left(\sum_{i=2}^{r}\left|x_i+\frac{\sum_{i=2}^r x_i}{n-r+1}\right|^p\right)\right\}. 
\end{align*} 
Then, in order to prove \eqref{grstar1234} enough to prove that 
\begin{align}\label{grborder1}
\sum_{i=2}^{r}|x_i|^p\le \left(1-\frac{1}{n}\right)^p\left(\sum_{i=2}^{r}|x_i|^p+\left|\sum_{i=2}^{r}x_i\right|^p\right),
\end{align}
and
\begin{align}\label{grborder2}
\sum_{i=2}^{r}|x_i|^p\le \left(1-\frac{1}{n}\right)^p\left(\sum_{i=2}^{r}\left|x_i+\frac{\sum_{i=2}^r x_i}{n-r+1}\right|^p\right),
\end{align}
for \eqref{grborder1} we observe that $\left|\displaystyle\sum_{i=2}^r x_i\right|^p\ge \max_{i=2,..r}|x_i|^p\ge \frac{\sum_{i=2}^r|x_i|^p}{r-1}$, so 
$$\left(1-\frac{1}{n}\right)^p\left(\sum_{i=2}^{r}|x_i|^p+\left|\sum_{i=2}^{r}x_i\right|^p\right)\ge \left(\sum_{i=2}^r|x_i|^p\right)\left(1-\frac{1}{n}\right)^p\left(1+\frac{1}{r-1}\right)\ge  \left(\sum_{i=2}^r|x_i|^p\right),$$
where we use that for $r\le n-1$ we have $$\left(1-\frac{1}{n}\right)^p\left(1+\frac{1}{r-1}\right)\ge \left(1-\frac{1}{n}\right)\left(1+\frac{1}{r-1}\right)\ge \left(1-\frac{1}{n}\right)\left(1+\frac{1}{n-1}\right)=1.$$ From this we conclude this inequality. \

For \eqref{grborder2}, we notice that $x_i+\frac{\sum_{i=2}^{r}x_i}{n-r+1}\ge x_i\left(1+\frac{1}{n-r+1}\right).$ Then, since (for $n\ge r\ge 2,$) we have  $$\left(1-\frac{1}{n}\right)^p\left(1+\frac{1}{n-r+1}\right)^p\ge \left(1-\frac{1}{n}\right)\left(1+\frac{1}{n-r+1}\right)\ge \left(1-\frac{1}{n}\right)\left(1+\frac{1}{n-1}\right)=1$$  we conclude this inequality, and therefore this case. Notice that this argument holds for every $p\in (0,1).$\\

{\it Case 2: $f(a_1)>m.$} Here, we observe that if $f(a_2)\leq f(a_1)$ then $|M_{S_n}f(a_1)-M_{S_n}f(a_i)|\le \frac{|f(a_1)-f(a_i)|}{2}$ for all $i\ge 2$ and thus \eqref{grinequalitystar} follows in this case (since $\var_{p}M_{S_n}f\le \frac{1}{2}\var_{p}f$). So we can assume that $f(a_2)>f(a_1).$ Let us take $k$ such that $f(a_2)\ge f(a_3)\ge \dots f(a_k)\ge f(a_1)> f(a_{k+1}),$ and $s$ is the minimum such that $f(a_1)+f(a_{s})\ge 2m.$ Let us keep in mind that, in this case, we have
$$(\var_{p}M_{S_n}f)^p=\displaystyle\sum_{i=2}^{k}|f(a_i)-f(a_1)|^p+\displaystyle\sum_{j=k+1}^{s}\left|\frac{f(a_1)-f(a_i)}{2}\right|^p+\displaystyle \sum_{i=s+1}^{n}\left|f(a_1)-m\right|^p.$$

Let us write $u=f(a_1)-m,$ $f(a_i)-f(a_1)=x_i$ for $i=2,..k$ and $y_i=f(a_1)-f(a_i)$ for $i=k+1,\dots n.$ We observe that, since $$m=\frac{\sum_{i=1}^nf(a_i)}{n}=\frac{\sum_{i=2}^k(m+u+x_i)+m+u+\sum_{i=k+1}^{n}(m+u-y_i)}{n},$$ we have $\displaystyle\sum_{i=2}^k x_i+nu=\displaystyle \sum_{k+1}^{n}y_i.$ 
Then \eqref{grinequalitystar} is equivalent to
\begin{align}\label{grequivalentstar}
\sum_{i=2}^kx_i^p+\sum_{i=k+1}^{s}\left(\frac{y_i}{2}\right)^p+\sum_{s+1}^{n}u^p\le \left(1-\frac{1}{n}\right)^p\left(\sum_{i=2}^kx_i^p+\sum_{i=k+1}^{n}y_i^p\right).
\end{align}
It is useful to solve first the case $k=n-1$ (observe, that then $s=n-1$). In this case we observe that $y_{n}=\displaystyle\sum_{i=2}^{k}x_i+nu.$ Then, we need to prove \begin{align}\label{grk=n-1}
\displaystyle\sum_{i=2}^kx_i^p+u^p\le \left(1-\frac{1}{n}\right)^p\left[\displaystyle\sum_{i=2}^k x_i^p+\left(\displaystyle\sum_{i=2}^k x_i+nu\right)^p\right].
\end{align}We notice first that, by Jensen's inequality we have 
$$\frac{(n-2)n^p+1}{n-1}\le \left(\frac{(n-2)n+1}{n-1}\right)^p=(n-1)^p,$$
then $(n-2)^{1-p}(n^p-(n-1)^p)\le (n-2)(n^p-(n-1)^p)\le (n-1)^p-1,$
where this last inequality is just another way of writing the previous claim. Then, by Jensen's inequality (in the second inequality), we have
\begin{align*}
((n-1)^p-1)\left(\sum_{i=2}^{k}x_i+nu\right)^p&\ge((n-1)^p-1)\left(\sum_{i=2}^{k}x_i\right)^p\\
&\ge ((n-1)^p-1)(n-2)^{p-1}\sum_{i=2}^{k}x_i^p\\
&\ge (n^p-(n-1)^p)\sum_{i=2}^{k}x_i^p, 
\end{align*}
where in the last inequality we use what we obtained before. 
 Then, $$(n-1)^p\left[\left(\sum_{i=2}^{k}x_i+nu\right)^p+\sum_{i=2}^{k}x_i^p\right]\ge n^p\sum_{i=2}^{k}x_i^p+\left(\sum_{i=2}^{k}x_i+nu\right)^p,$$
thus, we have 
\begin{align*}
 \sum_{i=2}^k(nx_i)^p+(nu)^p&\le \sum_{i=2}^k(nx_i)^p+\left(\sum_{i=2}^k x_i+nu\right)^p\\
 &\le (n-1)^p\left[\sum_{i=2}^kx_i^p+\left(\sum_{i=2}^k x_i+nu\right)^p\right],
\end{align*}
concluding the inequality \eqref{grk=n-1}.
 So, we assume in the following that $k\le n-2.$

We observe that $u\le \frac{y_i}{2}$ for $i=s+1,..n,$ and thus 
$$
\sum_{i=2}^kx_i^p+\sum_{i=k+1}^{s}\left(\frac{y_i}{2}\right)^p+\sum_{s+1}^{n}u^p\le \sum_{i=2}^kx_i^p+\sum_{i=k+1}^{n}\left(\frac{y_i}{2}\right)^p,$$
therefore \eqref{grequivalentstar} would follow if $$\sum_{i=2}^k x_i^p\left(1-\left(1-\frac{1}{n}\right)^p\right)\le \left(\left(1-\frac{1}{n}\right)^p-\frac{1}{2^p}\right)\left(\sum_{i=k+1}^{n} y_i^p\right).$$
 Indeed, by Jensen's inequality $\displaystyle\sum_{i=k+1}^{n} y_i^p\ge \left(\sum_{i=k+1}^{n}y_i\right)^p\ge \left(\sum_{i=2}^k x_i\right)^p\ge (k-1)^{p-1}\left(\sum_{i=2}^kx_i^p\right).$ So, we need $(k-1)^{1-p}\left[1-\left(1-\frac{1}{n}\right)^p\right]\le \left(1-\frac{1}{n}\right)^p-\frac{1}{2^p}.$
Since $k-1\le n-3$  is enough \begin{align}\label{grgoal}
(n-3)^{1-p}\left(1-\left(1-\frac{1}{n}\right)^p\right)\le \left(1-\frac{1}{n}\right)^p-\frac{1}{2^p},
\end{align}
but that is equivalent to $$(n-3)^{1-p}(n^p-(n-1)^p)\le (n-1)^p-\left(\frac{n}{2}\right)^p,$$ then, is enough to prove (we use here $n^p-(n-1)^p\le p(n-1)^{p-1}$ by the fundamental theorem of calculus)
\begin{align}\label{gr00}(n-3)^{1-p}p(n-1)^{p-1}\le (n-1)^p-\left(\frac{n}{2}\right)^p,\end{align}
or, the stronger bound (since $\left(\frac{n-3}{n-1}\right)^{1-p}\le 1$), $p\le (n-1)^p-(\frac{n}{2})^p.$ Fixed $p$, is possible to observe that this last inequality holds for $n$ big enough. Therefore, we conclude the last statement of Theorem \ref{theo 2} (iii). 
Now we assume that $1>p\ge \frac{1}{2}.$
First observe that for $n\ge 6$ we have that $p\le (n-1)^p-\left(\frac{n}{2}\right)^p,$ in fact $g(n)=(n-1)^p-\left(\frac{n}{2}\right)^p$ is increasing for $n\ge 2$ because its derivative is $p(n-1)^{p-1}-\frac{p}{2}\left(\frac{n}{2}\right)^{p-1}\ge 0$ since $2\ge 2^{1-p}\ge \left(\frac{2(n-1)}{n}\right)^{1-p}$ So, we need to prove $p\le 5^p-3^p$, indeed $g(p)=5^p-3^p-p$ is convex for $p\ge 0$ (its second derivative is $\log(5)^{2}5^p-\log(3)^{2}3^p\ge 0$) then since $g(0)=0$ and  $g\left(\frac{1}{2}\right)=\sqrt{5}-\sqrt{3}-\frac{1}{2}\ge 0$  for every $p\ge \frac{1}{2}$ we get $\alpha g(p)=\alpha g(p)+\beta g(0)\ge g\left(\frac{1}{2}\right)>0,$ for some $\alpha,\beta\ge 0.$ From where we conclude this inequality.   
Then, considering \cite[Theorem 1.4]{LX}, the only cases left are $n=4$ and $n=5.$ For $n=4$, considering \eqref{gr00}, we just need
\begin{align*}
\left(\frac{1}{3}\right)^{1-p}p\le 3^p-2^p,   
\end{align*}
or, equivalently, $p\le 3-3(\frac{2}{3})^p$, but $g(p)=3-3(\frac{2}{3})^p-p$ is concave in $(0,1),$ so, since $g(0)=0=g(1),$ we conclude in this case. Notice that this argument holds for every $1>p>0,$ and therefore the case $n=4$ is completed. 

Finally, for $n=5,$ we just need (considering \eqref{gr00}) $$\left(\frac{1}{2}\right)^{1-p}p\le 4^p-\left(\frac{5}{2}\right)^p,$$ or equivalently $$\frac{p}{2}\le 2^p-\left(\frac{5}{4}\right)^p,$$ but $g(p)=2^p-\left(\frac{5}{4}\right)^p-\frac{p}{2}$ is convex for $p\ge 0$ (because its second derivative is $\log(2)^{2}2^p-\log\left(\frac{5}{4}\right)^2\left(\frac{5}{4}\right)^p\ge 0$) then since $\sqrt{2}-\sqrt{\frac{5}{4}}-\frac{1}{4}\ge 0$ and $g(0)=0$ we conclude this case similarly as for $n\ge6$. Since we finish the analysis of cases, we conclude the proof of the theorem. 
\end{proof}
\begin{remark}
It is possible, in fact, to prove \eqref{grgoal} for every $0<p<1$ when $n=5,$ thus proving Theorem \ref{theo 2}(iii) for every $0<p<1$ in this case. We omit the details for the sake of simplicity.
\end{remark}\


\subsection{Qualitative results: Proof of Theorem \ref{theo 3}}
In the last part of this section we prove our versions of the qualitative results conjectured in  Conjecture C.

\begin{proof}[Proof of Theorem \ref{theo 3} (i)]
We assume without loss of generality that $f$ is non-negative. Also, in the following we assume that $G_n$ is connected, since the general case follows from there. 
Given $u,v\in G_n:=\{a_1,a_2,\dots,a_n\}$, such that $M_{\alpha,G_n}f(u)>M_{\alpha,G_n}f(v)$, we observe that there exists $k\leq n-1$ such that
$$
M_{\alpha,G_n}f(u)=\frac{|B(u,k)|^{\alpha}}{|B(u,k)|}\sum_{a_i\in B(u,k)}f(a_i),
$$
then
\begin{align*}
    M_{\alpha,G_n}f(u)-M_{\alpha,G_n}f(v)&\leq\frac{|B(u,k)|^{\alpha}}{|B(u,k)|}\sum_{a_i\in B(u,k)}f(a_i)-\frac{n^\alpha}{n}\sum_{i=1}^{n}f(a_i)\\
    &\leq n^{\alpha}\left[\frac{1}{|B(u,k)|}\sum_{a_i\in B(u,k)}f(a_i)-\frac{1}{n}\sum_{i=1}^{n}f(a_i)\right]\\
    &\leq n^{\alpha}(f(x)-f(y))\\
    &\leq n^{\alpha}(n-1)^{\max\{1-\frac{1}{p},0\}}\var_{p}f 
\end{align*}
Where, in the third line $x\in G_n$ is choose such that $f(x):=\max\{f(a_i);a_i\in B(u,k)\}$ and $y\in G_n$ is choose such that $f(y):=\min\{f(a_i);a_i\in G_n\}$. In the fourth line we used H\"older inequality.

Therefore 
\begin{align*}
\var_{q}M_{\alpha,G_n}&=\left(\frac{1}{2}\sum_{u\in G_n}\sum_{v\in N_{G_n}(u)}|M_{\alpha,G_n}f(u)-M_{\alpha,G_n}f(v)|^q\right)^{1/q}\\
&\leq \left(\frac{n(n-1)}{2}\right)^{1/q}n^{\alpha}(n-1)^{\max\{\frac{p-1}{p},0\}}\var_pf\\
&=C(n,p,q)\var_pf.
\end{align*}
\end{proof}

\begin{proof}[Proof of Theorem \ref{theo 3} (ii)]
We start observing that for all $j\geq 1$
\begin{align}\label{key obs in theo 3 (ii)}
    \|f-f_j\|_{l^\infty(G_n)}&=\max_{y\in V}|f(y)-f_j(y)|-\min_{x\in V}|f(x)-f_j(x)| +\min_{x\in V}|f(x)-f_j(x)|\nonumber\\
    &\leq \var(f-f_j)+\min_{x\in V}|f(x)-f_j(x)|\nonumber\\
    &\leq n^{\max\{1-1/p,0\}}\var_p(f-f_j)+\min_{x\in V}|f(x)-f_j(x)|.
\end{align}
Then, assuming that  $\lim_{j\to\infty}\min_{x\in V}|f(x)-f_j(x)|=0$, we have that
\begin{equation*}
    \|f-f_j\|_{l^\infty(G_n)}\to 0 \ \ \text{as} \ \ j\to\infty.
\end{equation*}
Moreover, for any $u,v\in G_n$ we have that
\begin{align*}
    M_{\alpha,G_n}f(u)-M_{\alpha,G_n}f_j(u)-[M_{\alpha,G_n}f(v)-M_{\alpha,G_n}f_j(v)]&\leq M_{\alpha,G_n}(f-f_j)(u)+M_{\alpha,G_n}(f-f_j)(v)\\
    &\leq 2\|f-f_j\|_{l^1(G_n)}\\
    &\leq 2n\|f-f_j\|_{l^{\infty}(G_n)}\to 0 \ \ \text{as}\ \ j\to\infty.
\end{align*}
Therefore
\begin{align*}
    \var_q(M_{\alpha,G_n}f-M_{\alpha,G_n}f_j)\leq \left(\frac{n(n-1)}{2}\right)^{1/q}2n\|f-f_j\|_{l^{\infty}(G_n)}\to 0 \ \ \text{as}\ \ j\to\infty.
\end{align*}

Finally, we observe that without the assumption that  $\lim_{j\to\infty}\min_{x\in V}|f(x)-f_j(x)|=0$ the continuity property could fail, with this purpose in mind consider the following situation: Let $G_n=S_n$ the star graph with $n$ vertices $V=\{a_1,a_2,\dots,a_n\}$ and center at $a_1$, for simplicity we take $\alpha=0$ and $p=q=1$. We define the function $f$ by $f(a_1)=2$ and $f(a_i)=1$ for all $i\neq 1$ thus $M_{S_n}f(a_1)=2$ and $M_{S_n}f(a_i)=3/2$ for all $i\neq1$. Then, we consider the sequence of functions $(f_{j})_{j\in\N}$ defined by $f_j(a_i)=f(a_i)-3$ for all $a_i\in V$ and for all $j\in \N$. Then $\var(f-f_j)=0$ for all $j\in\N$, moreover $M_{S_n}f_j(a_1)=\frac{1+2(n-1)}{n}$ and $M_{S_n}f_j(a_i)=2$ for all $i\neq 1$. Therefore 
\begin{align*}
\var(M_{S_n}f-M_{S_n}f_j)&\geq M_{S_n}f(a_1)-M_{S_n}f_j(a_1)-[M_{S_n}f(a_2)-M_{S_n}f_j(a_2)]\\
&=2-\frac{1+2(n-1)}{n}-[3/2-2]\\
&=\frac{1}{n}+\frac{1}{2}\ \ \text{for all}\ \ j\in\N.
\end{align*}
Then $\var(M_{S_n}f-M_{S_n}f_j)\nrightarrow 0$ as $j\to \infty$.

\end{proof}\

\begin{proof}[Proof of Theorem \ref{theo 3} (iii)] 
The boundedness follows using part (i) and the following inequality which is true for some $k\leq n-1$
\begin{align*}
    M_{\alpha,G_n}f(a_0)&=\frac{1}{|B(a_0,k)|^{1-\alpha}}\sum_{m\in B(a_0,k)}|f(m)|\\
    &=\frac{1}{|B(a_0,k)|^{1-\alpha}}\sum_{m\in B(a_0,k)}(|f(m)|-|f(a_0)|)+|B(a_0,k)|^\alpha|f(a_0)|\\
    &\leq |B(a_0,k)|^{\alpha}(\max_{m\in B(a_0,k)}|f(m)-f(a_0)|+|f(a_0)|)\\
    &\leq |B(a_0,k)|^\alpha(\var f+|f(a_0)|)\\
    &\leq |B(a_0,k)|^{\alpha}n^{\max\{1-1/p,0\}}(\var_pf+|f(a_0)|)\\
    &\leq n^{\alpha+\max\{1-1/p,0\}}\|f\|_{BV_p(G_n)}.
\end{align*}

The continuity follows using part (ii) and the following observations
$$
0\leq \var_p(f-f_j)+\min_{x\in V}|f(x)-f_j(x)| \leq \var_p(f-f_j)+|(f-f_j)(a_0)|=\|f-f_j\|_{BV_p(G_n)},
$$
and
\begin{align}
|M_{\alpha,G_n}f(a_0)-M_{\alpha,G_n}f_j(a_0)|&\leq M_{\alpha,G_n}(f-f_j)(a_0)\nonumber\\
&\leq \|f-f_j\|_{l^1(G_n)}\nonumber\\
&\leq n\|f-f_j\|_{l^\infty(G_n)}\nonumber\\
&\leq n^{1+\max\{1-1/p,0\}}\var_p(f-f_j)+n\min_{x\in V}|f(x)-f_j(x)|\nonumber\\
&\leq n^{1+\max\{1-1/p,0\}}\|f-f_j\|_{BV_p(G_n)},\nonumber
\end{align}
which is a consequence of \eqref{key obs in theo 3 (ii)}.
\end{proof}

\section{Proof of optimal bounds for the $2$-norm of maximal functions}
In this subsection we prove our results concerning the values $\|M_{G}\|_2$ for our graphs of interest. 
\subsection{2-norm of the maximal operator in $K_n$: Proof of Theorem \ref{p=2, complete graph} and Corollary \ref{p=2, n=3m}}
We start by proving that Corollary \ref{p=2, n=3m} follows by Theorem \ref{p=2, complete graph}.
\begin{proof}[Proof of Corollary \ref{p=2, n=3m}]
The inequality 
$$
\|M_{K_n}f\|_2\leq \left(\frac{4}{3}\right)^{1/2}\|f\|_2
$$
follows from the Theorem \ref{p=2, complete graph}, since $k=n/3$ in the right hand side. On the other hand, we consider the following example: we define $f:V\to\R$ by
$$
f(a_i)=4\ \text{for all}\ 1\leq i\leq \frac{n}{3} \ \text{and} \ f(a_i)=1\ \text{for all}\ \frac{n}{3}+1\leq i\leq n.
$$
Then, in this case we have
$$
M_{K_n}f(a_i)=4\ \text{for all} \ 1\leq i\leq \frac{n}{3} \ \text{and}\ M_{K_n}f(a_i)=2 \ \text{for all}\ \frac{n}{3}+1\leq i\leq n. 
$$
Therefore
\begin{equation*}
    \|M_{K_n}f\|_2=\left(\frac{\frac{16n}{3}+\frac{4(2n)}{3}}{\frac{16n}{3}+\frac{2n}{3}}\right)^{1/2}\|f\|_2=\left(\frac{4}{3}\right)^{1/2}\|f\|_2.
\end{equation*}
\end{proof}



Now we prove our bound that holds for $K_n$ for every $n\ge 2.$

\begin{proof}[Proof of Theorem \ref{p=2, complete graph}]
We assume without loss of generality that $f$ is no-negative. Consider the case
$$
f(a_1)\geq f(a_2)\geq\dots\geq f(a_k)\geq m\geq f(a_{k+1})\geq\dots\geq f(a_n).
$$
Then, in this case
$$
M_{K_n}f(a_i)=f(a_i) \ \text{for all}\ 1\leq i\leq k,\ \text{and}\ M_{K_n}f(a_i)=m \ \text{for all} \ k+1\leq i\leq n.
$$
Therefore 
\begin{align}\label{ineq:1 p=2 complete graph}
\|M_{K_n}f\|^{2}_2&=\sum_{i=1}^{k}f(a_i)^{2}+(n-k)m^2\nonumber\\&=\left(1+\frac{n-k}{n^2}\right)\sum_{i=1}^{k}f(a_i)^{2}+\frac{n-k}{n^2}\sum_{i=k+1}^{n}f(a_i)^2\nonumber\\
&\ \ \ +\frac{2(n-k)}{n^2}\sum_{\substack{1\leq i<j\leq k}}f(a_i)f(a_j) +\frac{2(n-k)}{n^2}\sum_{\substack{k+1\leq i<j\leq n}}f(a_i)f(a_j)\nonumber\\
&\ \ \ +\frac{2(n-k)}{n^2}\sum_{\substack{1\leq i\leq k\\ k+1\leq j\leq n}}f(a_i)f(a_j)\nonumber\\
&\leq \left(1+\frac{n-k}{n^2}\right)\sum_{i=1}^{k}f(a_i)^{2}+\frac{n-k}{n^2}\sum_{i=k+1}^{n}f(a_i)^2\\
&\ \ \ +\frac{(n-k)(k-1)}{n^2}\sum_{i=1}^{k}f(a_i)^2+\frac{(n-k)(n-k-1)}{n^2}\sum_{i=k+1}^{n}f(a_i)^2\nonumber\\
&\ \ \ +\frac{2(n-k)}{n^2}\sum_{\substack{1\leq i\leq k\\ k+1\leq j\leq n}}f(a_i)f(a_j)\nonumber\\
&=A_k\sum_{i=1}^{k}f(a_i)^{2}+B_K\sum_{i=k+1}^{n}f(a_i)^2+\frac{2(n-k)}{n^2}\sum_{\substack{1\leq i\leq k\nonumber\\ k+1\leq j\leq n}}f(a_i)f(a_j),
\end{align}
where $A_k:=1+\frac{(n-k)k}{n^2}$ and $B_k:=\frac{(n-k)^2}{n^2}$. Observe that $A_k-B_k=\frac{3nk-2k^2}{n^2}$ and by the AM-GM inequality
\begin{align}\label{ineq:2 p=2 complete graph}
    \|M_{K_n}f\|^{2}_2&\leq A_k\sum_{i=1}^{k}f(a_i)^{2}+B_k\sum_{i=k+1}^{n}f(a_i)^2+\frac{2(n-k)}{n^2}\sum_{\substack{1\leq i\leq k\\ k+1\leq j\leq n}}f(a_i)f(a_j)\nonumber\\
    &\leq A_k\sum_{i=1}^{k}f(a_i)^{2}+B_k\sum_{i=k+1}^{n}f(a_i)^2+\frac{1}{n^2}\sum_{\substack{1\leq i\leq k\\ k+1\leq j\leq n}}(xf(a_i)^2+yf(a_j)^2)\\
    &=\left(A_k+\frac{(n-k)x}{n^2}\right)\sum_{i=1}^{k}f(a_i)^{2}+\left(B_k+\frac{ky}{n^2}\right)\sum_{i=k+1}^{n}f(a_i)^2\nonumber
\end{align}
for all $0<x,y$ such that $xy=(n-k)^2$. Then, we choose $x,y$ such that 
$$
A_k+\frac{(n-k)x}{n^2}=B_k+\frac{ky}{n^2}.
$$
So, $x$ is the positive solution for the equation 
$$
(3nk-2k^2)x+(n-k)x^2=k(n-k)^2.
$$
More precisely
$$
x:=\frac{-(3nk-2k^2)+(4kn^3-3n^2k^2)^{1/2}}{2(n-k)}.
$$
Therefore, combining \eqref{ineq:1 p=2 complete graph} and \eqref{ineq:2 p=2 complete graph} we obtain
\begin{align*}
   \|M_{K_n}f\|^{2}_2&\leq\max_{k\in[1,n-1]}\left(A_k+\frac{(n-k)x}{n^2}\right)\sum_{i=1}^{n}f(a_i)^2\\
   &=\max_{k\in[1,n-1]}\left(1+\frac{(n-k)k}{n^2}+\frac{(4kn^3-3n^2k^2)^{1/2}-(3nk-2k^2)}{2n^2}\right)\sum_{i=1}^{n}f(a_i)^2\\
   &=\max_{k\in[1,n-1]}\left(1-\frac{k}{2n}+\frac{(4kn-3k^2)^{1/2}}{2n}\right)\sum_{i=1}^{n}f(a_i)^2.
\end{align*}
Then, we consider the function 
$g:[1,n-1]\to\R$ defined by $g(t):=-t+(4tn-3t^2)^{1/2}$.

Observe that 
$$
\max_{t\in[1,n-1]}g(t)=g\left(\frac{n}{3}\right).
$$
Moreover, $g$ is increasing in $[1,n/3]$ and decreasing in $[n/3,n-1]$. Therefore 
\begin{equation}\label{ineq:3 p=2 complete graph}
   \|M_{K_n}f\|^{2}_2\leq \max_{k\in\{\lfloor\frac{n}{3}\rfloor,\lceil\frac{n}{3}\rceil\}}\left(1-\frac{k}{2n}+\frac{(4kn-3k^2)^{1/2}}{2n}\right)\|f\|^2_2.
\end{equation}

Finally, observe that in order to have an equality in \eqref{ineq:3 p=2 complete graph} it is enough to have equality in \eqref{ineq:1 p=2 complete graph} and \eqref{ineq:2 p=2 complete graph}. Moreover, the equality in \eqref{ineq:1 p=2 complete graph} is attained if and only if $f(a_i)=f(a_1)=\gamma$ for all $1\leq i\leq k$, and $f(a_j)=f(a_{k+1})=\eta$ for all $k+1\leq j\leq n$, for some $0<\eta<\gamma$. We can assume without loss of generality that $\eta=1$.
On the other hand, the equality in \eqref{ineq:2 p=2 complete graph} is attained if and only if $y^{1/2}=x^{1/2}\gamma=(n-k)^{1/2}\gamma^{1/2}$, or equivalently $\gamma=\frac{n-k}{x}$. Therefore, in order to obtain an equality in \eqref{ineq:3 p=2 complete graph} for $k \in\{\lfloor\frac{n}{3}\rfloor,\lceil\frac{n}{3}\rceil\}$ we consider the function $g_{k}:V\to\R$ defined by
$$
g_{k}(a_i)=\gamma:=\frac{2(n-k)^2}{(4kn^3-3n^2k^2)^{1/2}-(3nk-2k^2)} \ \ \ \text{for all} \ \ 1\leq i \leq k,
$$
and $g_k(a_j)=1$ for all $k+1\leq j\leq n$. Then, by construction
$$
\|M_{K_n}\|_2=\max_{k\in\{\lfloor \frac{n}{3}\rfloor,\lceil \frac{n}{3}\rceil \}}\frac{\|M_{K_n}g_{k}\|_2}{\|g_k\|_2}.
$$
this shows that our bound is optimal, moreover we have found extremizers. Observe that, in the particular case when $n=3k$, we obtain $\gamma=4$ as in the Corollary \ref{p=2, n=3m}.

\end{proof}\



\subsection{2-norm of the maximal operator in $S_n$: Proof of Theorem \ref{p=2, star graph}.}
Now we prove our result concerning the $2$-norm of our maximal operator on $S_n.$

\begin{proof}[Proof of Theorem \ref{p=2, star graph}]
As usual we assume without loss of generality that $f$ is no negative and we denote by $m$ the average of $f$ along $V$ {\it{i.e.}} $m=\frac{\sum_{i=1}^{n}f(a_i)}{n}$. We observe that $M_{S_n}f(a_1)=f(a_1)$ or $M_{S_n}f(a_1)=m.$ We study this two cases separately.\\

{\it{Case 1: $M_{S_n}f(a_1)=f(a_1).$}}
Assume without loss of generality that $M_{S_n}f(a_i)=f(a_i)$ for all $1\leq i\leq k$, $M_{S_n}f(a_i)=\frac{f(a_i)+f(a_1)}{2}$ for all $k+1\leq i\leq k+r$, and $M_{S_n}f(a_i)=m$ for all $k+r+1\leq i\leq n$. By Cauchy-Schwarz inequality we have
$$
m^2\leq \frac{\sum_{i=1}^{n}f(a_i)^2}{n}.
$$
Then
\begin{align*}
    \|M_{S_n}f\|^2_2&\le \left(1+\frac{r}{4}\right)f(a_1)^2+\sum_{i=2}^{k}f(a_i)^2+\frac{1}{4}\sum_{i=k+1}^{k+r}f(a_i)^2+\frac{2}{4}\sum_{i=k+1}^{k+r}f(a_i)f(a_1) +\frac{s}{n}\sum_{i=1}^{n}f(a_i)^2\\
    &=\left(1+\frac{r}{4}+\frac{s}{n}\right)f(a_1)^2+\left(1+\frac{s}{n}\right)\sum_{i=2}^{k}f(a_i)^2+\left(\frac{1}{4}+\frac{s}{n}\right)\sum_{i=k+1}^{k+r}f(a_i)^2\\
    &\ \ \ +\frac{2}{4}\sum_{i=k+1}^{k+r}f(a_i)f(a_1) +\frac{s}{n}\sum_{i=k+r+1}^{n}f(a_i)^2.
\end{align*}
where $s:=n-k-r$. Moreover, for all $k+1\leq i\leq k+r$, we have that 
\begin{equation*}
    \frac{2}{4}f(a_i)f(a_1)\leq xf(a_1)^2+yf(a_i)^2
\end{equation*}
for all $x,y>0$ such that $xy\geq\frac{1}{16}$. We can choose $x$ and $y$ such that 
$$
y-rx=1+\frac{r-1}{4} \ \text{and} \ \ xy=\frac{1}{16}.
$$
or equivalently 
$$
x:=\frac{[(r+9)(r+1)]^{1/2}-(r+3)}{8r}.
$$
Therefore, for all $n\geq 4$ we have
\begin{align*}
    \|M_{S_n}f\|^2_2&\le \max_{\{k,r\in\N;1\leq k+r\leq n\}}\left(1+\frac{n-k-r}{n}+\frac{r}{4}+\frac{[(r+9)(r+1)]^{1/2}-(r+3)}{8}\right)\|f\|^2_2\\
    &\le \left(1+\frac{n-1}{4}+\frac{(n^2+8n)^{1/2}-(n+2)}{8}\right)\|f\|^2_2.
\end{align*}

{\it{Case 2: \ $M_{S_n}f(a_1)=m.$}} In this case $k\geq 2$. 
Following the same strategy (and notation), for all $n\geq 4$ we obtain that
\begin{align*}
    \|M_{S_n}f\|^2_2
    &\le \left(\frac{r}{4}+\frac{s+1}{n}\right)f(a_1)^2+\left(1+\frac{s+1}{n}\right)\sum_{i=2}^{k}f(a_i)^2+\left(\frac{1}{4}+\frac{s+1}{n}\right)\sum_{i=k+1}^{k+r}f(a_i)^2\\
    &\ \ \ +\frac{2}{4}\sum_{i=k+1}^{k+r}f(a_i)f(a_1) +\frac{s+1}{n}\sum_{i=k+r+1}^{n}f(a_i)^2.\\
    &\leq \max_{\{k,r\in\N;1\leq k+r\leq n\}}\left\{\frac{n-k-r+1}{n}+\frac{r+1}{4},\frac{n-k-r+1}{n}+1\right\}\|f\|^2_2\\
    &=\max_{\{k,r\in\N;1\leq k+r\leq n\}}\left\{\frac{n-k-r+1}{n}+\frac{r+1}{4},\frac{n-1}{n}+1\right\}\|f\|^2_2.
\end{align*}
The inequality 
$$\|M_{S_n}\|_2\le \left(1+\frac{n-1}{4}+\frac{(n^2+8n)^{1/2}-(n+2)}{8}\right)^{1/2}:=C_n$$ 

follows from these two estimates.\\ 

Finally, we observe that $\|M_{S_n}\|_2=C_n$. Consider the function $g:V\to\R$ defined by
$g(a_i)=1$ for all $1\leq i\leq n-1$ and $g(a_0)=\gamma$, where we choose
$ \gamma$ to be a positive real number larger than 1, such that $\gamma$ is a solution for the quadratic equation
$$
aX^2+bX+c:=\left( C^2_{n}-1-\frac{(n-1)}{4}\right)x^2-\frac{n-1}{2}x+C^2_{n}(n-1)-\frac{n-1}{4}=0.
$$
The existence of $\gamma$ follows from the definition of $C_n$, since we can see that $b^2-4ac=0$ and $\frac{-b}{2a}>1$. More precisely
$$
\gamma=-\frac{b}{2a}=\frac{2(n-1)}{(n^2+8n)^{1/2}-(n+2)}.
$$
For this particular function we have
$$
\frac{\|M_{S_n}g\|_2}{\|g\|_2}=\left(\frac{\gamma^2+(n-1)\left(\frac{\gamma+1}{2}\right)^2}{\gamma^2+(n-1)}\right)^{1/2}= C_n.
$$
This concludes the proof of our theorem.
\end{proof}
\section{Acknowledgements.}
The authors are thankful to Emanuel Carneiro, Terence Tao and the anonymous referees for very helpful comments. C.G.R was supported by CAPES-Brazil.




\bibliographystyle{amsplain}

\end{document}